\documentclass[11pt,reqno,a4paper]{amsart}

\usepackage{amsmath,amssymb,amsthm,mathtools}
\usepackage{bm}
\usepackage[colorlinks=true,citecolor=blue,linkcolor=blue,urlcolor=blue]{hyperref}
\usepackage{enumerate}
\usepackage{graphicx}
\usepackage[margin=1in]{geometry}

\theoremstyle{plain}
\newtheorem{theorem}{Theorem}[section]
\newtheorem{proposition}[theorem]{Proposition}
\newtheorem{lemma}[theorem]{Lemma}
\newtheorem{corollary}[theorem]{Corollary}
\theoremstyle{definition}
\newtheorem{definition}[theorem]{Definition}
\newtheorem{example}[theorem]{Example}

\theoremstyle{remark}
\newtheorem{remark}[theorem]{Remark}

\numberwithin{equation}{section}

\newcommand{\F}{\mathbb{F}}
\newcommand{\R}{\mathbb{R}}
\newcommand{\C}{\mathbb{C}}

\newcommand{\Sym}{\mathrm{Sym}}
\newcommand{\Ker}{\mathrm{Ker}}
\newcommand{\Img}{\mathrm{Im}}
\newcommand{\rank}{\mathrm{rank}}
\newcommand{\Span}{\mathrm{Span}}

\newcommand{\id}{\mathrm{id}}
\newcommand{\Hmacro}{\mathcal{H}_{\mathrm{macro}}}
\newcommand{\Vmacro}{V_{\mathrm{macro}}}
\newcommand{\cmacro}{c_{\mathrm{macro}}}
\newcommand{\Bmacro}{B_{\mathrm{macro}}}
\newcommand{\Zsp}{\mathcal{Z}}
\newcommand{\Lap}{L_{\beta}}
\newcommand{\Hbb}{\mathcal{H}}
\newcommand{\dphi}{\delta_{\phi}}
\newcommand{\hphi}{\hat{\phi}}

\begin{document}

\title[Defect Invariants and Gram Operators]{Defect Spaces and Gram Operators for Tensor-Valued Incidence Maps}

\author{Kengo Miyamoto}
\address{Department of Computer and Information Science, Ibaraki University, Ibaraki, 316-8511, Japan.}
\email{kengo.miyamoto.uz63@vc.ibaraki.ac.jp}

\subjclass[2020]{05C50, 05C65, 15A03, 15A72, 05B35}
\keywords{incidence map, cycle space, defect invariant, tensor algebra, edge Gram operator, oriented hypergraph}

\thanks{This work was supported by JSPS KAKENHI Grant Number 24K16885.}

\begin{abstract}
We study vector-valued incidence maps obtained from ordinary graph incidence maps by linear observation of the free vertex space.
Let $\F$ be a field, $D = (X, E, s, t)$ a finite directed multigraph, $U$ an $\F$-vector space, and $\phi : X \to U$ a vertex labeling with $\F$-linear extension $\hphi : \F^X \to U$.
The vector-valued incidence map $\partial_\phi : \F^E \to U$, $\partial_\phi(\mathbf{1}_e) = \phi(t(e)) - \phi(s(e))$, factors as $\partial_\phi = \hphi \circ B_D$, where $B_D$ is the classical incidence map of $D$. We prove the formula
$$\dim_\F \Ker(\partial_\phi) = |E| - |X| + c(D) + \delta_\phi,$$
where $c(D)$ is the number of weakly connected components of $D$ and $\delta_\phi := \dim_\F(\Img(B_D) \cap \Ker(\hphi))$ is the defect invariant.
We apply this framework to directed tensor-labeled hypergraphs $\Hbb = (Q_0, Q_1, \beta)$, in which each hyperedge carries a pair of boundary tensors $(A_e, B_e)$ in the tensor algebra $T(\F^{Q_0})$, and prove that $\delta(\Hbb) = 0$ over any field for each of the six standard constructions, including symmetric encodings that degenerate in positive characteristic.
Over $\F = \R$, the edge Gram operator $\Lap = \partial_\beta^* \partial_\beta$ has rank $|\Vmacro| - \cmacro - \delta(\Hbb)$, and its degree-truncated operators form a Loewner-monotone filtration whose rank increments equal the decrements of the defect filtration. We further realize the cycle space of every oriented hypergraph (in the sense of Reff--Rusnak) as $\Ker(\partial_\beta)$ within this framework, and exhibit a four-edge inclusion--exclusion example with $\delta(\Hbb) = 1$.
\end{abstract}

\maketitle

\section{Introduction}\label{sec:intro}

Let $\F$ be a field and let $D = (X, E, s, t)$ be a finite directed multigraph, that is, $X$ is a finite set of vertices, $E$ is a finite set of edges, and $s, t : E \to X$ assign to each edge $e \in E$ its source $s(e)$ and target $t(e)$. The \emph{incidence map} of $D$ is the $\F$-linear map
$$B_D : \F^E \longrightarrow \F^X, \quad B_D(\mathbf{1}_e) = \mathbf{1}_{t(e)} - \mathbf{1}_{s(e)},$$
whose kernel is the cycle space of $D$. The classical rank formula $\rank_\F B_D = |X| - c(D)$, where $c(D)$ is the number of weakly connected components of $D$, yields
\begin{equation}\label{eq:intro-classical-rank}
\dim_\F \Ker(B_D) = |E| - |X| + c(D).
\end{equation}
The identity \eqref{eq:intro-classical-rank} is a standard point of contact between algebraic graph theory \cite{Biggs, GodsilRoyle}, matroid theory \cite{Welsh, Oxley}, and the cellular homology of one-dimensional CW complexes \cite{Hatcher}.

The starting point of this paper is the following generalization of $B_D$. Let $U$ be an $\F$-vector space and let $\phi : X \to U$ be a vertex labeling. The associated \emph{vector-valued incidence map} is
\begin{equation}\label{eq:edge-diff}
\partial_\phi : \F^E \longrightarrow U, \quad \partial_\phi(\mathbf{1}_e) = \phi(t(e)) - \phi(s(e)).
\end{equation}
Write $\hphi : \F^X \to U$ for the $\F$-linear extension of $\phi$. Since $\partial_\phi = \hphi \circ B_D$, we have $\Ker(\partial_\phi) \supseteq \Ker(B_D)$. The inclusion is strict precisely when $\Img(B_D) \cap \Ker(\hphi) \neq 0$. The following identity quantifies the size of this enlargement and provides the organizing framework of the paper.

\begin{theorem}\label{thm:intro-main}
Let $\F$ be a field, $D = (X, E, s, t)$ a finite directed multigraph, $U$ an  $\F$-vector space, and $\phi : X \to U$ a map. Then
$$\dim_\F \Ker(\partial_\phi) = |E| - |X| + c(D) + \delta_\phi,$$
where $\delta_\phi := \dim_\F\left(\Img(B_D) \cap \Ker(\hphi)\right)$.
\end{theorem}

We call $\Delta_\phi := \Img(B_D) \cap \Ker(\hphi)$ the \emph{defect space} of $\phi$, and its dimension $\delta_\phi$ the \emph{defect invariant}. Equivalently, $\delta_\phi = \rank_\F(B_D) - \rank_\F(\partial_\phi)$, so $\delta_\phi$ is the rank drop caused by the passage from the free vertex space $\F^X$ to the label space $U$. When $\hphi$ is injective, $\delta_\phi = 0$ and Theorem~\ref{thm:intro-main} reduces to \eqref{eq:intro-classical-rank}. 
In Section~\ref{sec:general}, we further give a geometric description of $\rank_\F(\partial_\phi)$, an explicit basis of $\Ker(\partial_\phi)$ via spanning forests, and identify $\delta_\phi$ with the nullity of the linear matroid determined by rooted label differences.

The principal application of this framework is to a tensor-valued setting. 
Let $Q_0$ and $Q_1$ be two finite sets whose elements are called \emph{vertices} 
and \emph{hyperedges}, respectively, and let $T(\F^{Q_0})$ be the tensor algebra of the free vector space $\F^{Q_0}$. 
We call a tuple $\Hbb = (Q_0, Q_1, \beta)$, where
$$\beta : \F^{Q_1} \longrightarrow T(\F^{Q_0}) \times T(\F^{Q_0})$$
is an $\F$-linear map, a \emph{directed tensor-labeled hypergraph}. 
Set $\beta(\mathbf{1}_e)=(A_e,B_e)$. 
The associated \emph{tensor-valued incidence map} is the $\F$-linear map
$$\partial_\beta : \F^{Q_1} \longrightarrow T(\F^{Q_0}), \quad \partial_\beta(\mathbf{1}_e):=B_e-A_e,$$
and its kernel $\Zsp(\Hbb)$ is the \emph{tensor cycle space} of $\Hbb$.

The tensor-valued setting reduces to the vector-valued one through an auxiliary multigraph. Put $\Vmacro := \{A_e \mid e \in Q_1\} \cup \{B_e \mid e \in Q_1\} \subset T(\F^{Q_0})$ and form the directed multigraph $\Hmacro$ on $\Vmacro$ with edge set $Q_1$ and assignments $s(e) = A_e$, $t(e) = B_e$. We call $\Hmacro$ the \emph{associated macrograph} of $\Hbb$. 
The evaluation map $\hphi : \F^{\Vmacro} \to T(\F^{Q_0})$, $\mathbf{1}_w \mapsto w$, satisfies $\partial_\beta = \hphi \circ \Bmacro$, where $\Bmacro$ is the incidence map of $\Hmacro$. Theorem~\ref{thm:intro-main}, applied to $\Hmacro$ and $\hphi$, therefore gives
$$\dim_\F \Zsp(\Hbb) = |Q_1| - |\Vmacro| + \cmacro + \delta(\Hbb),$$
where $\cmacro := c(\Hmacro)$ and $\delta(\Hbb) := \dim_\F\left(\Img(\Bmacro) \cap \Ker(\hphi)\right)$.

Within this framework, ordinary directed and undirected graphs, multiset hyperedges, ordered-tuple hyperedges, and their directed analogues arise as six \emph{standard constructions} (Section~\ref{sec:tensor-hg}).

The central vanishing result of the paper is the following.

\begin{theorem}\label{thm:intro-vanishing}
Let $\Hbb$ be a directed tensor-labeled hypergraph in which every hyperedge is given by the same one of the six standard constructions of Section~\ref{sec:tensor-hg}. Then, over any field $\F$, $\delta(\Hbb) = 0$ holds.
\end{theorem}

By Theorem~\ref{thm:intro-vanishing}, the defect term in the dimension formula vanishes and $\dim_\F \Zsp(\Hbb)$ is determined by the combinatorial invariants $|Q_1|, |\Vmacro|, \cmacro$ of the associated macrograph over the fixed field $\F$. 
The non-triviality of Theorem~\ref{thm:intro-vanishing} lies in its validity in positive characteristic. 
For example, the symmetrized tensor
$$\Sym_k(v_\mu) := \sum_{\sigma \in S_k} v_{\sigma(1)} \otimes \cdots \otimes v_{\sigma(k)},$$
which encodes a multiset hyperedge of cardinality $k$, may vanish when some vertex multiplicity is at least the characteristic. We show that this degeneration contributes to $\Ker(\Bmacro)$, called the \emph{topological cycle space} $\Zsp_{\mathrm{top}}(\Hbb)$, but not to the \emph{algebraic cycle space} $\Zsp_{\mathrm{alg}}(\Hbb) := \Img(\Bmacro) \cap \Ker(\hphi)$ measured by $\delta(\Hbb)$. We remark that, outside the standard constructions, $\delta(\Hbb) > 0$ can occur. Minimal examples with positive defect are given in Sections~\ref{sec:defect} and~\ref{sec:related}.

Beyond this vanishing, the kernel $\Zsp(\Hbb)$ admits a hierarchy of approximations indexed by tensor degree. For any $\F$-linear map $\rho : T(\F^{Q_0}) \to U'$, called an \emph{observation map}, we put
$$\Zsp_\rho(\Hbb) := \Ker(\rho \circ \partial_\beta), \quad \delta_\rho(\Hbb) := \dim_\F\left(\Img(\Bmacro) \cap \Ker(\rho \circ \hphi)\right).$$
The same defect formalism gives the dimension formula
$$\dim_\F \Zsp_\rho(\Hbb) = |Q_1| - |\Vmacro| + \cmacro + \delta_\rho(\Hbb).$$
The degree-truncation maps $\pi_{\leq k} : T(\F^{Q_0}) \to T^{\leq k}(\F^{Q_0})$ produce a natural filtration
$$\Zsp(\Hbb) = \Zsp_{\leq K}(\Hbb) \subseteq \Zsp_{\leq K-1}(\Hbb) \subseteq \cdots \subseteq \Zsp_{\leq 0}(\Hbb),$$
where $K$ is the maximal tensor degree appearing in $\partial_\beta$, and the successive quotients are realized as images of the homogeneous components of $\partial_\beta$. We remark that, over a field $\F$ of characteristic $2$, the projection onto the first tensor factor recovers the classical $\F_2$-coefficient cycle space of an undirected graph from its symmetric-tensor encoding (Section~\ref{sec:observation}, Theorem~\ref{thm:F2-recovery}).

This filtration admits a spectral counterpart. Let $\F = \R$ and equip $T(\R^{Q_0})$ with the standard inner product making the standard tensor basis orthonormal. Write $\partial_\beta^*$ for the adjoint of $\partial_\beta$, and put
$$\Lap := \partial_\beta^* \partial_\beta : \R^{Q_1} \longrightarrow \R^{Q_1},$$
which we call the \emph{edge Gram operator} of $\Hbb$.

\begin{theorem}\label{thm:intro-gram}
Over $\F = \R$, the edge Gram operator $\Lap$ is symmetric positive semi-definite with $\Ker(\Lap) = \Zsp(\Hbb)$ and
$$\rank(\Lap) = |\Vmacro| - \cmacro - \delta(\Hbb).$$
The degree-truncated operators $L_{\leq k} := (\pi_{\leq k} \partial_\beta)^* (\pi_{\leq k} \partial_\beta)$ form a Loewner-monotone chain
$$L_{\leq 0} \preceq L_{\leq 1} \preceq \cdots \preceq L_{\leq K} = \Lap,$$
and their rank jumps coincide with the drops in the defect filtration:
$$\rank(L_{\leq k}) - \rank(L_{\leq k-1}) = \delta_{\leq k-1}(\Hbb) - \delta_{\leq k}(\Hbb).$$
\end{theorem}

We remark that for a loopless simple graph in Construction~(1) with $m = |Q_1|$ edges, $\Lap = 2 I_m + J_m$ depends only on $m$, so the Gram operator does not distinguish adjacency structure beyond the edge count. Finer information appears in the presence of parallel edges, loops, or symmetric-tensor degenerations (Section~\ref{sec:spectrum}).

The framework connects naturally to the oriented hypergraph theory of Reff--Rusnak \cite{ReffRusnak, Rusnak, ChenLiuRobinsonRusnakWang}. An \emph{oriented hypergraph} is a triple $\Hbb_o = (Q_0, Q_1, \mathbb{B}^{\mathrm{oh}})$ in which $\mathbb{B}^{\mathrm{oh}} \in \{-1, 0, +1\}^{Q_0 \times Q_1}$ is the incidence matrix.

\begin{theorem}\label{thm:intro-oh}
Every oriented hypergraph $\Hbb_o$ admits a canonical directed tensor-labeled hypergraph $F(\Hbb_o)$ with $A_e, B_e \in T^1(\F^{Q_0}) = \F^{Q_0}$ satisfying
$$\Zsp(F(\Hbb_o)) = \Ker(\mathbb{B}^{\mathrm{oh}}).$$
Moreover, in the star-shaped case, $\delta(F(\Hbb_o))$ is the affine-dependence defect of the indicator vectors of the terminal vertex sets (Proposition~\ref{prop:oh-star-minimality}); in particular, the inclusion--exclusion identity $\mathbf{1}_X + \mathbf{1}_Y = \mathbf{1}_{X \cup Y} + \mathbf{1}_{X \cap Y}$ yields a minimal example with $r = 4$ hyperedges and $\delta = 1$.
\end{theorem}

The paper is structured as follows. In Section~\ref{sec:general} we prove Theorem~\ref{thm:intro-main} and develop the general theory of vector-valued incidence maps, including the spanning-forest basis and the matroid-nullity interpretation of $\delta_\phi$. In Section~\ref{sec:tensor-hg} we introduce directed tensor-labeled hypergraphs, the six standard constructions, and the dimension formula for $\Zsp(\Hbb)$. Section~\ref{sec:defect} proves Theorem~\ref{thm:intro-vanishing} and exhibits examples of positive defect outside the standard class. In Section~\ref{sec:observation} we develop observation maps, observed cycle spaces, and degree filtrations. In Section~\ref{sec:spectrum} we study Gram operators and prove Theorem~\ref{thm:intro-gram}. In Section~\ref{sec:related} we prove Theorem~\ref{thm:intro-oh} and relate the framework to oriented hypergraph theory.

Throughout the paper, $\F$ denotes an arbitrary field. Specific characteristic assumptions are stated explicitly where required.

\section{General theory of vector-valued incidence maps}\label{sec:general}

\subsection{Basic setup}

Let $D = (X, E, s, t)$ be a finite directed multigraph, that is, $X$ is a finite set of vertices, $E$ is a finite set of edges, and $s, t : E \to X$ assign to each edge $e \in E$ its source $s(e)$ and target $t(e)$. We allow loops and parallel edges. 
We write $c(D)$ for the number of weakly connected components of $D$.

\begin{definition}
Let $\F^X$ and $\F^E$ be the free $\F$-vector spaces on $X$ and $E$, with standard basis vectors $\mathbf{1}_x$ ($x \in X$) and $\mathbf{1}_e$ ($e \in E$). The \textbf{incidence matrix} of $D$ is the $\F$-linear map $B_D : \F^E \to \F^X$ defined on basis vectors by
$$B_D(\mathbf{1}_e) := \mathbf{1}_{t(e)} - \mathbf{1}_{s(e)} \quad (e \in E)$$
and extended by $\F$-linearity.
\end{definition}

The following is classical (see \cite[Proposition~4.3]{Biggs}).

\begin{lemma}\label{lem:classical-rank}
For any field $\F$,
\begin{equation}\label{eq:classical-rank}
\rank_{\F}(B_D)=|X|-c(D),\quad \dim_{\F} \Ker(B_D) = |E| - |X| + c(D).
\end{equation}
\end{lemma}

\begin{definition}\label{def:edge-diff}
Let $U$ be an $\F$-vector space and $\phi : X \to U$ a map, which we call a \textbf{vector-valued labeling}. Denote by $\hphi : \F^X \to U$, $\mathbf{1}_x \mapsto \phi(x)$, the $\F$-linear extension of $\phi$. The \textbf{vector-valued incidence map} $\partial_{\phi} : \F^E \to U$ is the $\F$-linear map defined on basis vectors by
$$\partial_{\phi}(\mathbf{1}_e) := \phi(t(e)) - \phi(s(e)) \quad (e \in E)$$
and extended by $\F$-linearity.
\end{definition}

\begin{proposition}\label{prop:factorization}
The equation $\partial_{\phi} = \hphi \circ B_D$ holds. 
\end{proposition}

\begin{proof}
By the definitions, $\hphi(B_D(\mathbf{1}_e)) = \hphi(\mathbf{1}_{t(e)} - \mathbf{1}_{s(e)}) = \phi(t(e)) - \phi(s(e)) = \partial_{\phi}(\mathbf{1}_e)$. As both sides are $\F$-linear and agree on the basis $\{\mathbf{1}_e\}_{e \in E}$, the equality holds on $\F^E$.
\end{proof}

\subsection{The defect invariant and the kernel dimension formula}

\begin{definition}\label{def:defect}
The \textbf{defect invariant} of a vector-valued labeling $\phi : X \to U$ is
$$\dphi := \dim_{\F}\left(\Img(B_D) \cap \Ker(\hphi)\right).$$
\end{definition}

\begin{remark}
The subspace $\Img(B_D) \subset \F^X$ is the space of edge boundaries of $D$, that is, the elements whose coordinate sum on each weakly connected component equals zero. The subspace $\Ker(\hphi) \subset \F^X$ encodes the $\F$-linear relations among the labels $\phi(x)$ in $U$. Thus $\dphi$ measures the dimension of the linear dependencies among labels that are detected by the boundary operator.
\end{remark}

\begin{theorem}\label{thm:rank-nullity}
Let $\F$ be a field, $D = (X, E, s, t)$ a finite directed multigraph, $U$ an $\F$-vector space, and $\phi : X \to U$ a labeling. Then
\begin{equation}\label{eq:main}
\dim_{\F} \Ker(\partial_{\phi}) = |E| - |X| + c(D) + \dphi.
\end{equation}
\end{theorem}

\begin{proof}
By Proposition \ref{prop:factorization}, $\partial_{\phi} = \hphi \circ B_D$. Hence $\xi \in \Ker(\partial_{\phi})$ if and only if $B_D(\xi) \in \Ker(\hphi)$. The restriction of $B_D$ to $\Ker(\partial_{\phi})$ is an $\F$-linear surjection onto $\Img(B_D) \cap \Ker(\hphi)$ with kernel $\Ker(B_D)$, which yields the short exact sequence
$$0 \to \Ker(B_D) \to \Ker(\partial_{\phi}) \xrightarrow{B_D} \Img(B_D) \cap \Ker(\hphi) \to 0.$$
Thus, we have $\dim_{\F} \Ker(\partial_{\phi}) = \dim_{\F} \Ker(B_D) + \dphi$. Substituting Lemma \ref{lem:classical-rank} implies \eqref{eq:main}.
\end{proof}

In particular, if $\hphi$ is injective then $\Ker(\hphi) = 0$, so $\dphi = 0$, and \eqref{eq:main} reduces to the classical formula \eqref{eq:classical-rank}.

\begin{corollary}\label{cor:defect-rank-drop}
The defect invariant equals the rank drop
$$\dphi = \rank_{\F}(B_D) - \rank_{\F}(\partial_{\phi}) = \left(|X| - c(D)\right) - \rank_{\F}(\partial_{\phi}).$$
\end{corollary}

\begin{proof}
By Proposition \ref{prop:factorization}, $\Img(\partial_{\phi}) = (\hphi \circ B_D)(\F^E) = \hphi(\Img(B_D))$. Applying the rank-nullity theorem to the restriction $\hphi|_{\Img(B_D)} : \Img(B_D) \to U$, we obtain
\begin{align*}
\rank_{\F}(B_D) &=  \dim_{\F} \hphi(\Img(B_D)) + \dim_{\F}\left(\Img(B_D) \cap \Ker(\hphi)\right) \\
&= \rank_{\F}(\partial_{\phi}) + \dphi.
\end{align*}
The claim now follows from $\rank_{\F}(B_D) = |X| - c(D)$ in Lemma \ref{lem:classical-rank}.
\end{proof}

\subsection{Geometric interpretation of the rank}

\begin{proposition}\label{prop:geometric-rank}
For each weakly connected component $C \subset X$ of $D$, choose a basepoint $r_C \in C$ and set
$$U_C := \Span_{\F}\{\phi(x) - \phi(r_C) \mid x \in C\} \subset U.$$
Then, we have
\begin{equation}\label{eq:rank-sum}
\rank_{\F}(\partial_{\phi}) = \dim_{\F}\left(\sum_C U_C\right),
\end{equation}
where the sum is over the weakly connected components of $D$.
\end{proposition}

\begin{proof}
We have $\Img(\partial_{\phi}) = \Span_{\F}\{\phi(t(e)) - \phi(s(e)) \mid e \in E\}$. For an edge $e$ in the component $C$,
$$\phi(t(e)) - \phi(s(e)) = \left(\phi(t(e)) - \phi(r_C)\right) - \left(\phi(s(e)) - \phi(r_C)\right) \in U_C.$$
Conversely, let $x \in C$ and take a path $r_C = x_0, x_1, \ldots, x_{\ell} = x$ from $r_C$ to $x$ in the underlying graph of $D$. For each $i$, choose an edge $e_i \in E$ joining $x_{i-1}$ and $x_i$, and put
$$\varepsilon_i := \begin{cases} +1 & \text{if } s(e_i) = x_{i-1}, \\ -1 & \text{if } s(e_i) = x_i. \end{cases}$$
Then $\varepsilon_i \partial_{\phi}(\mathbf{1}_{e_i}) = \phi(x_i) - \phi(x_{i-1})$ in both cases. Hence
$$\partial_{\phi}\left(\sum_{i=1}^{\ell} \varepsilon_i \mathbf{1}_{e_i}\right) = \sum_{i=1}^{\ell} \left(\phi(x_i) - \phi(x_{i-1})\right) = \phi(x) - \phi(r_C),$$
so $\phi(x) - \phi(r_C) \in \Img(\partial_{\phi})$. We conclude $\Img(\partial_{\phi}) = \sum_C U_C$.
\end{proof}

\begin{proposition}\label{prop:affine-defect}
For each weakly connected component $C$ of $D = (X, E, s, t)$, fix a basepoint $r_C \in C$ and define the $\F$-linear map
$$A_{\phi} : \bigoplus_C \F^{C \setminus \{r_C\}} \longrightarrow U, \quad A_{\phi}(\mathbf{1}_x) := \phi(x) - \phi(r_C) \quad (x \in C \setminus \{r_C\}).$$
Then the following hold.
\begin{enumerate}
\item[\textup{(1)}] $\dphi = \dim_{\F} \Ker(A_{\phi})$, and this value is independent of the choice of basepoints $\{r_C\}$.
\item[\textup{(2)}] $\dphi = 0$ if and only if the family $\left(\phi(x) - \phi(r_C)\right)_{C, x \in C \setminus \{r_C\}}$ is $\F$-linearly independent in $U$.
\item[\textup{(3)}] If $D$ is connected, then $\dphi = 0$ if and only if the family $(\phi(x))_{x \in X}$ is affinely independent in $U$.
\end{enumerate}
\end{proposition}

\begin{proof}
(1) We have $\Img(A_{\phi}) = \Span_{\F}\{\phi(x) - \phi(r_C)\} = \sum_C U_C$, so by Proposition \ref{prop:geometric-rank}, $\rank_{\F}(A_{\phi}) = \dim_{\F}\left(\sum_C U_C\right) = \rank_{\F}(\partial_{\phi})$. The dimension of the domain is $\sum_C (|C| - 1) = |X| - c(D)$, so Corollary \ref{cor:defect-rank-drop} gives
$$\dim_{\F} \Ker(A_{\phi}) = \left(|X| - c(D)\right) - \rank_{\F}(A_{\phi}) = \left(|X| - c(D)\right) - \rank_{\F}(\partial_{\phi}) = \dphi.$$
If $r_C$ is replaced by another basepoint $r_C'$, then
$$\phi(x) - \phi(r_C') = \left(\phi(x) - \phi(r_C)\right) - \left(\phi(r_C') - \phi(r_C)\right),$$
which shows that $U_C$ is independent of the choice of basepoint. Hence $\rank_{\F}(A_{\phi})$ is independent of $\{r_C\}$, and so is $\dim_{\F} \Ker(A_{\phi}) = \dphi$.

(2) The injectivity of $A_{\phi}$ is equivalent to the $\F$-linear independence of $\{A_{\phi}(\mathbf{1}_x)\}_{C, x \in C \setminus \{r_C\}}$. The claim follows from (1).

(3) This is immediate from (2) and the definition of affine independence.
\end{proof}

For subspaces of $U$ we have in general
$$\dim_{\F}\left(\sum_C U_C\right) \leq \sum_C \dim_{\F} U_C,$$
with equality if and only if the sum is direct. By Corollary \ref{cor:defect-rank-drop} and Proposition \ref{prop:geometric-rank},
\begin{equation}\label{eq:geometric-defect}
\dphi = \left(|X| - c(D)\right) - \dim_{\F}\left(\sum_C U_C\right).
\end{equation}
For each component $C$, writing $B_{D,C}$ and $\hphi_C$ for the restrictions, Corollary \ref{cor:defect-rank-drop} applied to $C$ gives $\dim U_C = (|C| - 1) - \dim_{\F}(\Img(B_{D,C}) \cap \Ker(\hphi_C))$. Substituting this into \eqref{eq:geometric-defect}, we obtain the decomposition
\begin{equation}
\dphi = \underbrace{\left[\sum_C \dim_{\F} U_C - \dim_{\F}\left(\sum_C U_C\right)\right]}_{\text{cross-component coupling}} + \underbrace{\sum_C \dim_{\F}\left(\Img(B_{D,C}) \cap \Ker(\hphi_C)\right)}_{\text{within-component algebraic dependencies}}.
\end{equation}
The first term, which vanishes precisely when the sum $\sum_C U_C$ is direct, measures the dimension in which labels from different components are linearly coupled in $U$. The second measures, for each component, the linear relations among the labels that are compatible with the boundary. We note that $\dphi$ is not a local quantity determined componentwise, but a global quantity that depends on the configuration of all labels in $U$.

\begin{remark}\label{rem:matroid-nullity}
Proposition \ref{prop:affine-defect} yields a matroid-theoretic interpretation of $\dphi$. Let $\mathcal{A}_{\phi} := (\phi(x) - \phi(r_C))_{C, x \in C \setminus \{r_C\}}$ be the family of rooted differences, and let $M[\mathcal{A}_{\phi}]$ be the linear matroid on $\mathcal{A}_{\phi}$ in $U$. Then
$$\dphi = |\mathcal{A}_{\phi}| - \rank(\mathcal{A}_{\phi}),$$
which is the corank of $M[\mathcal{A}_{\phi}]$. Although the matroid $M[\mathcal{A}_{\phi}]$ itself depends on the choice of basepoints $\{r_C\}$, its nullity $\dphi$ does not (Proposition \ref{prop:affine-defect}~(1)). In the tensor version developed in Section \ref{sec:tensor-hg}, $\delta(\Hbb)$ is reinterpreted as the corank of the linear matroid defined by the family $(w - r_C)_{C, w \in C \setminus \{r_C\}}$ (Corollary \ref{cor:tensor-affine-defect}).
\end{remark}

\subsection{Basis construction via spanning forests}

A \textbf{spanning forest} of $D$ is a subset $T \subset E$ such that the underlying graph of the sub-multigraph of $D$ with edge set $T$ contains no cycle and satisfies $|T| = |X| - c(D)$.

For two vertices $a, b$ lying in the same weakly connected component, the underlying graph of $T$ contains a unique simple path from $a$ to $b$. Write this path as $a = y_0, y_1, \ldots, y_m = b$, and let $g_k \in T$ be the tree edge joining $y_{k-1}$ and $y_k$. Set $\eta_k := +1$ if $s(g_k) = y_{k-1}$, and $\eta_k := -1$ if $s(g_k) = y_k$. The \textbf{signed path vector} from $a$ to $b$ is
$$F[a, b] := \sum_{k=1}^{m} \eta_k \mathbf{1}_{g_k} \;\in\; \F^T $$
(with $F[a, a] := 0$, taking $m = 0$).

\begin{definition}\label{def:spanning-forest}
Let $T$ be a spanning forest of $D$. For each edge $e \in E \setminus T$, the \textbf{topological cycle} at $e$ is
$$Z_e^{(\mathrm{top})} := \mathbf{1}_e - F[s(e), t(e)] \;\in\; \F^E.$$
\end{definition}

The signed path vector satisfies
\begin{equation}\label{eq:path-telescope}
B_D(F[a, b]) = \mathbf{1}_b - \mathbf{1}_a.
\end{equation}
Indeed, for the path $a = y_0, \ldots, y_m = b$ above, the definition of $\eta_k$ gives $\eta_k B_D(\mathbf{1}_{g_k}) = \eta_k(\mathbf{1}_{t(g_k)} - \mathbf{1}_{s(g_k)}) = \mathbf{1}_{y_k} - \mathbf{1}_{y_{k-1}}$ in either case. Hence
$$B_D(F[a, b]) = \sum_{k=1}^{m} \eta_k B_D(\mathbf{1}_{g_k}) = \sum_{k=1}^{m} (\mathbf{1}_{y_k} - \mathbf{1}_{y_{k-1}}) = \mathbf{1}_b - \mathbf{1}_a.$$

\begin{lemma}\label{lem:topological-cycle}
The family $\{Z_e^{(\mathrm{top})}\}_{e \in E \setminus T}$ is an $\F$-basis of $\Ker(B_D)$.
\end{lemma}

\begin{proof}
For each $e \in E \setminus T$, by the definition of $B_D$ and \eqref{eq:path-telescope},
$$B_D(Z_e^{(\mathrm{top})}) = B_D(\mathbf{1}_e) - B_D(F[s(e), t(e)]) = (\mathbf{1}_{t(e)} - \mathbf{1}_{s(e)}) - (\mathbf{1}_{t(e)} - \mathbf{1}_{s(e)}) = 0,$$
so $Z_e^{(\mathrm{top})} \in \Ker(B_D)$.

By Definition \ref{def:spanning-forest}, $F[s(e), t(e)]$ has support in $T$, so the projection $\pi : \F^E \to \F^{E \setminus T}$ onto the non-tree coordinates satisfies $\pi(F[s(e), t(e)]) = 0$. 
Hence $\pi(Z_e^{(\mathrm{top})}) = \pi(\mathbf{1}_e) = \mathbf{1}_e$ for $e \in E \setminus T$. 
If $\sum_{e \in E \setminus T} \alpha_e Z_e^{(\mathrm{top})} = 0$ with $\alpha_e \in \F$, then applying $\pi$ gives $\sum_{e \in E \setminus T} \alpha_e \mathbf{1}_e = 0$, and since $\{\mathbf{1}_e\}_{e \in E \setminus T}$ is the standard basis of $\F^{E \setminus T}$, all $\alpha_e = 0$.

The dimension of $\Span_{\F}\{Z_e^{(\mathrm{top})}\}_{e \in E \setminus T}$ is $|E \setminus T| = |E| - |X| + c(D)$, which by Lemma \ref{lem:classical-rank} equals $\dim_{\F} \Ker(B_D)$. 
Therefore, $\{Z_e^{(\mathrm{top})}\}_{e \in E \setminus T}$ is an $\F$-basis of $\Ker(B_D)$.
\end{proof}

\begin{lemma}\label{lem:algebraic-lift}
For any $r \in \Img(B_D) \cap \Ker(\hphi) \subset \F^X$, there exists $\zeta_r \in \Ker(\partial_{\phi})$ such that $B_D(\zeta_r) = r$.
\end{lemma}

\begin{proof}
Write $r = \sum_{x \in X} r_x \mathbf{1}_x$ with $r_x \in \F$. 
Note that $\Img(B_D)$ is spanned by the elements $B_D(\mathbf{1}_e) = \mathbf{1}_{t(e)} - \mathbf{1}_{s(e)}$.  
For any $e\in E$ $s(e)$ and $t(e)$  lie in the same weakly connected component. 
Thus, $\sum_{x \in C} r_x = 0$ for each weakly connected component $C$.

For each component $C$, fix a basepoint $r_C \in C$. Using the signed path vector $F[r_C, x] \in \F^T \subset \F^E$ for $x \in C \setminus \{r_C\}$, define
$$\zeta_r := \sum_C \sum_{x \in C \setminus \{r_C\}} r_x\, F[r_C, x] \;\in\; \F^E.$$
By \eqref{eq:path-telescope}, $B_D(F[r_C, x]) = \mathbf{1}_x - \mathbf{1}_{r_C}$. The condition $\sum_{x \in C} r_x = 0$ gives $\sum_{x \in C \setminus \{r_C\}} r_x = -r_{r_C}$, and so
\begin{align*}
B_D(\zeta_r) &= \sum_C \sum_{x \in C \setminus \{r_C\}} r_x \left(\mathbf{1}_x - \mathbf{1}_{r_C}\right) \\
&= \sum_C \left(\sum_{x \in C \setminus \{r_C\}} r_x \mathbf{1}_x - \left(\sum_{x \in C \setminus \{r_C\}} r_x\right) \mathbf{1}_{r_C}\right) \\
&= \sum_C \left(\sum_{x \in C \setminus \{r_C\}} r_x \mathbf{1}_x + r_{r_C} \mathbf{1}_{r_C}\right) = \sum_{x \in X} r_x \mathbf{1}_x = r.
\end{align*}
Furthermore, by Proposition \ref{prop:factorization} and $r \in \Ker(\hphi)$, we have $\partial_{\phi}(\zeta_r) =\hphi(r) = 0$. Hence $\zeta_r \in \Ker(\partial_{\phi})$ and $B_D(\zeta_r) = r$.
\end{proof}

\begin{theorem}\label{thm:basis}
Fix $D = (X, E, s, t)$, $\phi : X \to U$, and a spanning forest $T$. Let $\{r_1, \ldots, r_{\dphi}\}$ be any $\F$-basis of $\Zsp_{\mathrm{alg}} := \Img(B_D) \cap \Ker(\hphi)$, and let $\{\zeta_{r_1}, \ldots, \zeta_{r_{\dphi}}\}$ be lifts as in Lemma \ref{lem:algebraic-lift}. Then the set
\begin{equation}\label{eq:basis-extended}
\mathcal{B}_T^{\mathrm{ext}} := \{Z_e^{(\mathrm{top})}\}_{e \in E \setminus T} \cup \{\zeta_{r_1}, \ldots, \zeta_{r_{\dphi}}\}
\end{equation}
is an $\F$-basis of $\Ker(\partial_{\phi})$.
\end{theorem}

\begin{proof}
Every element of $\mathcal{B}_T^{\mathrm{ext}}$ lies in $\Ker(\partial_{\phi})$ by Proposition \ref{prop:factorization}, Lemmas \ref{lem:topological-cycle} and \ref{lem:algebraic-lift}.  
We show $\mathcal{B}_T^{\mathrm{ext}}$ is linearly independent. Suppose
$$\sum_{e \in E \setminus T} \alpha_e Z_e^{(\mathrm{top})} + \sum_{i=1}^{\dphi} \gamma_i \zeta_{r_i} = 0, \quad \alpha_e, \gamma_i \in \F.$$
Applying $B_D$ and using $Z_e^{(\mathrm{top})} \in \Ker(B_D)$ together with $B_D(\zeta_{r_i}) = r_i$, we obtain $\sum_{i=1}^{\dphi} \gamma_i r_i = 0$. Since $\{r_1, \ldots, r_{\dphi}\}$ is an $\F$-basis of $\Zsp_{\mathrm{alg}}$, all $\gamma_i = 0$. Substituting back, $\sum_{e \in E \setminus T} \alpha_e Z_e^{(\mathrm{top})} = 0$, and by Lemma \ref{lem:topological-cycle}, all $\alpha_e = 0$.

The topological cycles in $\mathcal{B}_T^{\mathrm{ext}}$ are indexed by $E \setminus T$ and the lifts by $\{1, \ldots, \dphi\}$, and the linear independence above shows all these elements are distinct. By the definition of a spanning forest, $|E \setminus T| = |E| - |X| + c(D)$, so $|\mathcal{B}_T^{\mathrm{ext}}| = |E| - |X| + c(D) + \dphi$. By Theorem \ref{thm:rank-nullity}, this equals $\dim_{\F} \Ker(\partial_{\phi})$. Hence $\mathcal{B}_T^{\mathrm{ext}}$ is an $\F$-basis of $\Ker(\partial_{\phi})$.
\end{proof}

If $U = \F^X$ and $\phi(x) = \mathbf{1}_x$ (the identity labeling), then $\hphi = \id_{\F^X}$, $\Ker(\hphi) = 0$, and $\dphi = 0$. The formula \eqref{eq:main} reduces to the classical formula \eqref{eq:classical-rank}, and Theorem \ref{thm:basis} recovers the classical fundamental cycle basis theorem associated with a spanning forest (\cite[Ch.~14, Section 2]{GodsilRoyle}).

\section{Directed tensor-labeled hypergraphs and the tensor incidence operator}\label{sec:tensor-hg}

In this section we define directed tensor-labeled hypergraphs and apply to them the general theory of Section \ref{sec:general}. The notion subsumes the classical directed hypergraph framework of Gallo, Longo, Nguyen, and Pallottino \cite{GalloLongoNguyenPallottino} and admits a uniform encoding of multiset, ordered, and tensor-valued hyperedge data.

\subsection{Directed tensor-labeled hypergraphs}\label{sec:tensor-prep}

Let $V$ be an $\F$-vector space and $T^k(V) := V^{\otimes k}$ ($T^0(V) := \F$) its $k$-th tensor power. The $\F$-algebra
$$T(V) := \bigoplus_{k \geq 0} T^k(V),$$
equipped with the tensor product $\otimes$ as multiplication, is called the \textbf{tensor algebra} of $V$. When $V = \F^{Q_0}$, the set
\begin{equation}\label{eq:standard-basis}
\mathcal{B} := \{1\} \cup \{u_1 \otimes \cdots \otimes u_k \mid k \geq 1, \; u_i \in Q_0\}
\end{equation}
is an $\F$-basis of $T(\F^{Q_0})$, which we call the \textbf{standard basis}.

\begin{remark}\label{rem:word-space}
The non-unit elements $u_1 \otimes \cdots \otimes u_k$ of $\mathcal{B}$ are in bijection with the words of length $k$ on $Q_0$, and under this bijection $T(\F^{Q_0})$ is $\F$-algebra isomorphic to the free associative algebra $\F\langle Q_0 \rangle$ (the tensor product corresponding to the concatenation of words). However, the proofs of the main results in this paper rely only on the graded decomposition $T(\F^{Q_0}) = \bigoplus_k T^k(\F^{Q_0})$ and on the standard basis $\mathcal{B}$, not on the algebra structure. The word notation for pure tensors is adopted only for convenience. We identify each $v \in Q_0$ with the corresponding standard basis vector $\mathbf{1}_v \in \F^{Q_0} = T^1(\F^{Q_0})$. Under this identification, the non-unit element $u_1 \otimes \cdots \otimes u_k$ of $\mathcal{B}$ denotes the tensor product $\mathbf{1}_{u_1} \otimes \cdots \otimes \mathbf{1}_{u_k} \in T^k(\F^{Q_0})$.
\end{remark}

\begin{definition}
For $k \geq 1$, the \textbf{symmetrization operator} $\Sym_k : T^k(V) \to T^k(V)$ is defined by
\begin{equation}\label{eq:Sym}
\Sym_k(m_1 \otimes \cdots \otimes m_k) := \sum_{\sigma \in S_k} m_{\sigma(1)} \otimes \cdots \otimes m_{\sigma(k)}.
\end{equation}
For $k = 0$, we set $\Sym_0 := \id_{\F}$.
\end{definition}

\begin{lemma}\label{lem:sym-invariance}
For any $\tau \in S_k$ and any $m_1, \ldots, m_k \in V$,
$$\Sym_k(m_{\tau(1)} \otimes \cdots \otimes m_{\tau(k)}) = \Sym_k(m_1 \otimes \cdots \otimes m_k).$$
\end{lemma}

\begin{proof}
Since $\sigma \mapsto \tau \sigma$ is a bijection of $S_k$,
\begin{align*}
\Sym_k(m_{\tau(1)} \otimes \cdots \otimes m_{\tau(k)})
&= \sum_{\sigma \in S_k} m_{\tau\sigma(1)} \otimes \cdots \otimes m_{\tau\sigma(k)} \\
&= \sum_{\sigma' \in S_k} m_{\sigma'(1)} \otimes \cdots \otimes m_{\sigma'(k)} \\
&= \Sym_k(m_1 \otimes \cdots \otimes m_k).
\end{align*}
Thus, the assertion follows.
\end{proof}

We note that, in positive characteristic, factorial coefficients appearing in the value of $\Sym_k$ may vanish in $\F$. For instance, if $\mathrm{char}(\F) = 2$, then $\Sym_2(v \otimes v) = 2(v \otimes v) = 0$. A precise characterization of the vanishing condition is given in Proposition \ref{prop:sym-vanishing}~(2).

\begin{definition}\label{def:tensor-hg}
Let $Q_0 = \{v_1, \ldots, v_n\}$ be a finite set of vertices and $Q_1 = \{e_1, \ldots, e_m\}$ a finite set of directed hyperedges. A \textbf{directed tensor-labeled hypergraph} is a triple $\Hbb = (Q_0, Q_1, \beta)$, where
\begin{equation}\label{eq:incidence-assignment}
\beta : \F^{Q_1} \longrightarrow T(\F^{Q_0}) \times T(\F^{Q_0})
\end{equation}
is an $\F$-linear map, called the \textbf{incidence assignment} of $\Hbb$. We abbreviate the image of the basis vector $\mathbf{1}_e$ ($e \in Q_1$) by $\beta(e) := \beta(\mathbf{1}_e) = (A_e, B_e)$, and call $A_e \in T(\F^{Q_0})$ the \textbf{source tensor} and $B_e \in T(\F^{Q_0})$ the \textbf{target tensor} of $e$. The map $\beta$ is completely determined by the data $\{(A_e, B_e)\}_{e \in Q_1}$.
\end{definition}

\subsection{Standard constructions (1)--(6)}

Ordinary graphs and various hypergraphs arise as special cases of Definition \ref{def:tensor-hg}. In what follows, we assume the non-emptiness of the structural data ($k \geq 1$ in (3) and (4); $p, q \geq 1$ in (5) and (6)). Undirected structures are encoded by placing the unit $1 \in T^0(\F^{Q_0})$ of the target tensor.

\begin{enumerate}
\item[(1)] \textbf{Symmetric quadratic encoding of undirected edges.}\footnote{This encoding is not intended to reproduce the classical undirected cycle space directly. Rather, it records a symmetric quadratic boundary label. To recover the classical undirected cycle space one may either choose an orientation for each edge and use Construction (2), or, in characteristic $2$, apply the observation map of Theorem~\ref{thm:F2-recovery}.} Let $Q_0$ be the vertex set and $Q_1$ the edge set, and let $\psi : Q_1 \to 2^{Q_0}$ satisfy $1 \leq |\psi(e)| \leq 2$. For each $e \in Q_1$, set
$$\beta(e) := \begin{cases} (2(v \otimes v), 1) & \text{if } \psi(e) = \{v\} \text{ (a loop)}, \\ 
(u \otimes v + v \otimes u, 1) & \text{if } \psi(e) = \{u, v\}, \; u \neq v. \end{cases}$$

\item[(2)] \textbf{Ordinary directed graph.} Let $D = (Q_0, Q_1, s, t)$ be a directed graph with source $s : Q_1 \to Q_0$ and target $t : Q_1 \to Q_0$. For each $e \in Q_1$, set $\beta(e) := (s(e), t(e))$.

\item[(3)] \textbf{Undirected hypergraph with multiset hyperedges.} Let $\psi$ assign to each $e \in Q_1$ a multiset $\psi(e)$ on $Q_0$, and put $k = |\psi(e)|$. If $\psi(e) = \{\!\{u_1, \ldots, u_k\}\!\}$, set $\beta(e) := (\Sym_k(u_1 \otimes \cdots \otimes u_k), 1)$. By Lemma \ref{lem:sym-invariance}, $\Sym_k(u_1 \otimes \cdots \otimes u_k)$ depends only on the multiset $\psi(e)$. Construction (1) is the special case of (3) obtained by viewing an ordinary edge $\{u, v\}$ ($u \neq v$) as the multiset $\{\!\{u, v\}\!\}$ and a loop $\{v\}$ as the multiset $\{\!\{v, v\}\!\}$.

\item[(4)] \textbf{Hypergraph with ordered tuple hyperedges.} Let $\psi$ assign to each $e \in Q_1$ an ordered tuple $\psi(e) = (u_1, \ldots, u_k)$ of vertices, and set $\beta(e) := (u_1 \otimes \cdots \otimes u_k, 1)$.

\item[(5)] \textbf{Directed hypergraph with multiset components.} Let $\psi_s$ and $\psi_t$ assign to each $e \in Q_1$ a source multiset $\psi_s(e)$ and a target multiset $\psi_t(e)$ on $Q_0$, and put $p = |\psi_s(e)|$, $q = |\psi_t(e)|$. If $\psi_s(e) = \{\!\{u_1, \ldots, u_p\}\!\}$ and $\psi_t(e) = \{\!\{v_1, \ldots, v_q\}\!\}$, set
$$\beta(e) := (\Sym_p(u_1 \otimes \cdots \otimes u_p), \Sym_q(v_1 \otimes \cdots \otimes v_q)).$$

\item[(6)] \textbf{Directed hypergraph with ordered components.} Let $\psi$ assign to each $e \in Q_1$ a source tuple $(u_1, \ldots, u_p)$ and a target tuple $(w_1, \ldots, w_q)$, and set $\beta(e) := (u_1 \otimes \cdots \otimes u_p, w_1 \otimes \cdots \otimes w_q)$. Construction (2) is the case $p = q = 1$.
\end{enumerate}

In Construction (2), $\beta(e) = (s(e), t(e)) \in T^1(\F^{Q_0}) \times T^1(\F^{Q_0})$, so the difference $B_e - A_e$ takes values in $T^1(\F^{Q_0}) = \F^{Q_0}$ and recovers the classical incidence matrix. 
In Construction (1), every hyperedge is encoded in $T^2(\F^{Q_0})$. When $k \geq 2$ in Constructions (3) or (4), or $\max(p, q) \geq 2$ in (5) or (6), the internal structure of a hyperedge is recorded in higher tensors and the difference $B_e - A_e$ involves higher-degree components. 
The cases $k = 1$ in (3) or (4) and $(p, q) = (1, 1)$ in (5) or (6), on the other hand, take values in $T^1(\F^{Q_0})$ and include the situation isomorphic to Construction (2). 
The cycle space $\Zsp(\Hbb)$ of the tensor incidence operator $\partial_\beta$ introduced in the next subsection (Definition \ref{def:tensor-incidence}), in these constructions, captures the algebraic relations among the tensor labels.

We say that $\Hbb$ follows a \textbf{single standard construction} if every hyperedge of $\Hbb$ is given by the same standard construction. Mixed constructions are not treated in this paper.

\begin{remark}\label{rem:divided-power}
Constructions (3) and (5) encode multiset hyperedges through $\Sym_k$. In characteristic zero, $\mathrm{Im}(\Sym_k)$ coincides with the symmetric tensor space $T^k(\F^{Q_0})^{S_k}$. In positive characteristic, by Proposition \ref{prop:sym-vanishing}~(2), $\Sym_k$ may vanish when a vertex multiplicity is at least the characteristic, so the inclusion $\mathrm{Im}(\Sym_k) \subsetneq T^k(\F^{Q_0})^{S_k}$ can be strict; for example, in characteristic $2$ with $V = \F v_1 \oplus \F v_2$, $T^2(V)^{S_2}$ has dimension $3$ while $\mathrm{Im}(\Sym_2)$ is the $1$-dimensional space $\F(v_1 \otimes v_2 + v_2 \otimes v_1)$.

A characteristic-free alternative is the divided power algebra \cite{Roby} $\Gamma(\F^{Q_0}) = \bigoplus_{k \geq 0} \Gamma^k(\F^{Q_0})$, where $\Gamma^k(\F^{Q_0})$ has basis $v_1^{[m_1]} \cdots v_r^{[m_r]}$ ($v_i \in Q_0$ distinct, $\sum_i m_i = k$). The assignment sending $v_1^{[m_1]} \cdots v_r^{[m_r]}$ to the sum of all distinct arrangements of the corresponding multiset gives an isomorphism $\Gamma^k(\F^{Q_0}) \xrightarrow{\sim} T^k(\F^{Q_0})^{S_k}$, under which
$$\Sym_k(v_1^{\otimes m_1} \otimes \cdots \otimes v_r^{\otimes m_r}) = \left(\prod_i m_i!\right) v_1^{[m_1]} \cdots v_r^{[m_r]}$$
holds in any characteristic. The two formulations coincide in characterisitc zero and diverge precisely where $\prod_i m_i!$ vanishes.

We adopt the symmetric-tensor formulation in Constructions (3) and (5). The vanishing theorem of Section \ref{sec:defect} asserts $\delta(\Hbb) = 0$ even under the positive-characteristic degeneration in which $\Vmacro$ (see, Definition \ref{def:macro}) contains zero tensors; since this degeneration does not occur in the divided-power encoding, the non-triviality of the theorem is specific to the symmetric tensor encoding.
\end{remark}

\subsection{The defect invariant and the tensor dimension formula}

Let $\Hbb = (Q_0, Q_1, \beta)$ be a directed tensor-labeled hypergraph.

\begin{definition}\label{def:tensor-incidence}
The $\F$-linear map $\partial_{\beta} : \F^{Q_1} \to T(\F^{Q_0})$ defined by
\begin{equation}\label{eq:tensor-incidence}
\partial_{\beta}(\mathbf{1}_e) := B_e - A_e \quad (e \in Q_1)
\end{equation}
is called the \textbf{tensor-valued incidence map} of $\Hbb$, also referred to as the tensor incidence operator. Its kernel $\Zsp(\Hbb) := \Ker(\partial_{\beta})$ is called the \textbf{tensor cycle space} of $\Hbb$.
\end{definition}

\begin{definition}\label{def:macro}
The subset
$$\Vmacro := \{A_e \mid e \in Q_1\} \cup \{B_e \mid e \in Q_1\} \subset T(\F^{Q_0})$$
is called the \textbf{set of boundary tensors} of $\Hbb$. The directed multigraph
$$\Hmacro := (\Vmacro, Q_1, s_{\mathrm{macro}}, t_{\mathrm{macro}}), \quad s_{\mathrm{macro}}(e) := A_e, \quad t_{\mathrm{macro}}(e) := B_e,$$
with vertex set $\Vmacro$ and edge set $Q_1$, is called the \textbf{associated macrograph} of $\Hbb$. We write $\cmacro$ for the number of weakly connected components of $\Hmacro$, and $\Bmacro : \F^{Q_1} \to \F^{\Vmacro}$ for the incidence matrix of $\Hmacro$.
\end{definition}

\begin{definition}\label{def:evaluation}
The $\F$-linear map $\hphi : \F^{\Vmacro} \to T(\F^{Q_0})$ defined by $\mathbf{1}_w \mapsto w$ ($w \in \Vmacro$) is called the \textbf{evaluation map}. Its image is denoted $W_{\Hbb} := \Span_{\F}(\Vmacro)$.
\end{definition}

\begin{proposition}\label{prop:tensor-factor}
The equation $\partial_{\beta} = \hphi \circ \Bmacro$ holds.
\end{proposition}

\begin{proof}
Apply Proposition \ref{prop:factorization} with $D = \Hmacro$, $\phi(w) = w$, $U = T(\F^{Q_0})$.
\end{proof}

\begin{definition}\label{def:tensor-defect}
The \textbf{defect invariant} of $\Hbb$ is
$$\delta(\Hbb) := \dim_{\F}\left(\Img(\Bmacro) \cap \Ker(\hphi)\right).$$
\end{definition}

\begin{definition}\label{def:cycle-decomposition}
For a directed tensor-labeled hypergraph $\Hbb$, the subspaces
$$\Zsp_{\mathrm{top}}(\Hbb) := \Ker(\Bmacro) \subseteq \F^{Q_1}, \quad \Zsp_{\mathrm{alg}}(\Hbb) := \Img(\Bmacro) \cap \Ker(\hphi)$$
are called the \textbf{topological cycle space} and the \textbf{algebraic cycle space} of $\Hbb$, respectively.
\end{definition}

\begin{proposition}\label{prop:cycle-ses}
The inclusion $\Zsp_{\mathrm{top}}(\Hbb) \subseteq \Zsp(\Hbb)$ holds. Moreover,  the restriction of $\Bmacro$ yields the short exact sequence
$$0 \longrightarrow \Zsp_{\mathrm{top}}(\Hbb) \longrightarrow \Zsp(\Hbb) \xrightarrow{\Bmacro} \Zsp_{\mathrm{alg}}(\Hbb) \longrightarrow 0.$$
In particular, $\Zsp(\Hbb) / \Zsp_{\mathrm{top}}(\Hbb) \cong \Zsp_{\mathrm{alg}}(\Hbb)$.
\end{proposition}

\begin{proof}
It follows from Proposition \ref{prop:tensor-factor} that 
\[ \Zsp_{\mathrm{top}}(\Hbb) = \Ker(\Bmacro) \subseteq \Ker(\hphi \circ \Bmacro) = \Ker(\partial_{\beta}) = \Zsp(\Hbb). \]
The restriction $\Bmacro|_{\Zsp(\Hbb)} : \Zsp(\Hbb) \to \F^{\Vmacro}$ has kernel $\Zsp_{\mathrm{top}}(\Hbb)$, and image $\Bmacro(\Ker(\hphi \circ \Bmacro)) = \Img(\Bmacro) \cap \Ker(\hphi) = \Zsp_{\mathrm{alg}}(\Hbb)$. The claim follows from the first isomorphism theorem.
\end{proof}

The quotient $\Zsp(\Hbb) / \Zsp_{\mathrm{top}}(\Hbb)$ is, via Proposition \ref{prop:cycle-ses}, canonically isomorphic to the subspace $\Zsp_{\mathrm{alg}}(\Hbb) \subseteq \F^{\Vmacro}$; however, the corresponding lift to a subspace of $\F^{Q_1}$ depends on the choice of a spanning forest (cf.\ Lemma \ref{lem:algebraic-lift}) and is therefore non-canonical. When we speak of ``algebraic cycles in $\Zsp(\Hbb)$'' in this paper, the canonical object is the quotient.

\begin{lemma}\label{lem:Bmacro-image}
For $v = \sum_{w \in \Vmacro} v_w \mathbf{1}_w \in \F^{\Vmacro}$ with $v_w \in \F$, $v \in \Img(\Bmacro)$ holds if and only if
\[ \sum_{w \in C} v_w = 0 \]
for every weakly connected component $C$ of $\Hmacro$.
\end{lemma}

\begin{proof}
Since $A_e$ and $B_e$ lie in the same weakly connected component of $\Hmacro$, each generator $\Bmacro(\mathbf{1}_e) = \mathbf{1}_{B_e} - \mathbf{1}_{A_e}$ satisfies the stated condition, and so does every element of $\Img(\Bmacro)$. By Lemma \ref{lem:classical-rank}, $\rank_{\F}(\Bmacro) = |\Vmacro| - \cmacro$, which coincides with the dimension of the subspace defined by the condition. The inclusion is therefore an equality.
\end{proof}

\begin{theorem}\label{thm:tensor-dim}
For any directed tensor-labeled hypergraph $\Hbb$, we have
\begin{equation}\label{eq:tensor-dim}
\dim_{\F} \Zsp(\Hbb) = |Q_1| - |\Vmacro| + \cmacro + \delta(\Hbb).
\end{equation}
\end{theorem}

\begin{proof}
Apply Theorem \ref{thm:rank-nullity} together with Proposition \ref{prop:tensor-factor}, with $D = \Hmacro$, $X = \Vmacro$, $E = Q_1$, $U = T(\F^{Q_0})$, and $\phi(w) = w$.
\end{proof}

\begin{proposition}\label{prop:delta-iso}
The defect invariant $\delta(\Hbb)$ depends only on the isomorphism class of $\Hbb$. That is, if $\Hbb = (Q_0, Q_1, \beta)$ and $\Hbb' = (Q_0', Q_1', \beta')$ admit bijections $f : Q_0 \to Q_0'$ and $g : Q_1 \to Q_1'$ such that the induced tensor algebra isomorphism $T(f) : T(\F^{Q_0}) \to T(\F^{Q_0'})$ satisfies
$$(T(f) \times T(f)) \circ \beta = \beta' \circ g,$$
where $g$ also denotes its $\F$-linear extension $\F^{Q_1} \to \F^{Q_1'}$, then $\delta(\Hbb) = \delta(\Hbb')$.
\end{proposition}

\begin{proof}
The bijections $f$ and $g$ induce $\F$-linear isomorphisms $\F^{\Vmacro} \cong \F^{\Vmacro'}$ and $\F^{Q_1} \cong \F^{Q_1'}$ intertwining $\Bmacro$ with $B_{\mathrm{macro}}'$ and $\hphi$ with $\hat\phi'$. Hence the subspace $\Img(\Bmacro) \cap \Ker(\hphi)$ is mapped isomorphically onto $\Img(B_{\mathrm{macro}}') \cap \Ker(\hat\phi')$, and their dimensions agree.
\end{proof}

\begin{corollary}\label{cor:tensor-defect-rank-drop}
Let $\Hbb$ be a directed tensor-labeled hypergraph. Then
$$\delta(\Hbb) = \left(|\Vmacro| - \cmacro\right) - \rank_{\F}(\partial_{\beta}).$$
\end{corollary}

\begin{proof}
Apply Corollary \ref{cor:defect-rank-drop} with $D = \Hmacro$ and $\partial_{\phi} = \partial_{\beta}$. By Lemma \ref{lem:classical-rank}, $\rank_{\F}(\Bmacro) = |\Vmacro| - \cmacro$.
\end{proof}

\begin{corollary}\label{cor:tensor-affine-defect}
For each weakly connected component $C$ of $\Hmacro$, fix a basepoint tensor $r_C \in \Vmacro$ and define the $\F$-linear map
$$A_{\Hbb} : \bigoplus_C \F^{C \setminus \{r_C\}} \longrightarrow T(\F^{Q_0}), \quad A_{\Hbb}(\mathbf{1}_w) := w - r_C.$$
Then $\delta(\Hbb) = \dim_{\F} \Ker(A_{\Hbb})$, and $\delta(\Hbb) = 0$ if and only if the family $(w - r_C)_{C, w \in C \setminus \{r_C\}}$ is $\F$-linearly independent in $T(\F^{Q_0})$. In particular, if $\Hmacro$ is connected, then $\delta(\Hbb) = 0$ if and only if $\Vmacro$ is affinely independent in $T(\F^{Q_0})$.
\end{corollary}

\begin{proof}
Apply Proposition \ref{prop:affine-defect} with $D = \Hmacro$, $U = T(\F^{Q_0})$, and $\phi(w) = w$.
\end{proof}

\begin{theorem}\label{thm:tensor-basis}
Fix a spanning forest $T \subset Q_1$ of $\Hmacro$. Let $\{r_1, \ldots, r_{\delta}\}$, where $\delta = \delta(\Hbb)$, be any $\F$-basis of $\Zsp_{\mathrm{alg}}(\Hbb)$, and let $\{\zeta_{r_1}, \ldots, \zeta_{r_{\delta}}\}$ be the corresponding lifts provided by Lemma \ref{lem:algebraic-lift}. Then the set
\begin{equation}
\mathcal{B}_T^{\mathrm{ext}} := \{Z_e^{(\mathrm{top})}\}_{e \in Q_1 \setminus T} \cup \{\zeta_{r_1}, \ldots, \zeta_{r_{\delta}}\}
\end{equation}
is an $\F$-basis of $\Zsp(\Hbb)$.
\end{theorem}

\begin{proof}
Apply Theorem \ref{thm:basis} directly.
\end{proof}

\section{Vanishing of the defect invariant for standard constructions}\label{sec:defect}

The short exact sequence
$$0 \to \Zsp_{\mathrm{top}}(\Hbb) \to \Zsp(\Hbb) \to \Zsp_{\mathrm{alg}}(\Hbb) \to 0$$
of Proposition \ref{prop:cycle-ses} decomposes $\Zsp(\Hbb)$ into a topological component $\Zsp_{\mathrm{top}}(\Hbb)$, of dimension $|Q_1| - |\Vmacro| + \cmacro$, and an algebraic component $\Zsp_{\mathrm{alg}}(\Hbb)$, of dimension $\delta(\Hbb)$. The main result of this section is Theorem \ref{thm:defect-vanishing}, which asserts that $\delta(\Hbb) = 0$ for any $\Hbb$ following a single standard construction over any field $\F$. Even when symmetric tensors degenerate in positive characteristic, this effect is confined to $\Zsp_{\mathrm{top}}(\Hbb)$ and does not appear in the algebraic part. In Section \ref{subsec:nontrivial-defect} we conversely give examples where $\delta(\Hbb) > 0$ arises from linear-combination labels beyond the standard constructions.

By Corollary~\ref{cor:tensor-affine-defect}, the vanishing of $\delta(\Hbb)$ is equivalent to the $\F$-linear independence of the rooted differences $(w - r_C)_{C, w \in C \setminus \{r_C\}}$ in $T(\F^{Q_0})$. For each standard construction we establish a stronger property---that $\Vmacro \setminus \{0\}$ itself is $\F$-linearly independent (Lemma~\ref{lem:Vmacro-nonzero-indep})---and then deduce $\delta(\Hbb) = 0$ after handling the possibility that the zero tensor belongs to $\Vmacro$, which may occur in positive characteristic.

\subsection{The main vanishing theorem}

\begin{proposition}\label{prop:sym-vanishing}
Let $\mu = \{\!\{v_1, \ldots, v_k\}\!\}$ be a non-empty multiset on $Q_0$, and for each $v \in Q_0$ let $m_v := |\{i \mid v_i = v\}|$ denote the multiplicity of $v$ in $\mu$. Set $v_\mu := v_1 \otimes \cdots \otimes v_k \in T^k(\F^{Q_0})$, and write
$$\mathcal{B}_\mu := \{w_1 \otimes \cdots \otimes w_k \in \mathcal{B} \mid \{\!\{w_1, \ldots, w_k\}\!\} = \mu\} \subset T^k(\F^{Q_0})$$
for the set of standard basis elements whose underlying multiset is $\mu$. Then the following hold.
\begin{enumerate}
\item[\textup{(1)}] The equation
\begin{equation}\label{eq:sym-formula}
\Sym_k(v_\mu) = \left(\prod_v m_v!\right) \sum_{w \in \mathcal{B}_\mu} w
\end{equation}
holds.
\item[\textup{(2)}] $\Sym_k(v_\mu) = 0$ if and only if $\mathrm{char}(\F) = p > 0$ and $m_v \geq p$ for some $v \in Q_0$.
\item[\textup{(3)}] The subfamily $\{\Sym_{|\mu|}(v_\mu)\}_\mu$, indexed by $\mu$ such that $\Sym_{|\mu|}(v_\mu) \neq 0$, is $\F$-linearly independent in $T(\F^{Q_0})$.
\end{enumerate}
\end{proposition}

\begin{proof}
(1) For each $w \in \mathcal{B}_\mu$, the number of $\sigma \in S_k$ with $v_{\sigma(1)} \otimes \cdots \otimes v_{\sigma(k)} = w$ equals the order of the Young subgroup $\prod_v S_{m_v}$, which is $\prod_v m_v!$. The formula \eqref{eq:sym-formula} follows.

(2) Since $\sum_{w \in \mathcal{B}_\mu} w$ is a sum of distinct standard basis elements with coefficient $1$, it is non-zero in $T^k(\F^{Q_0})$. Hence $\Sym_k(v_\mu) = 0$ if and only if $\prod_v m_v! = 0$ in $\F$. If $\mathrm{char}(\F) = 0$, this never holds. If $\mathrm{char}(\F) = p > 0$, then since $p$ is prime, $p \mid \prod_v m_v!$ if and only if $p \mid m_v!$ for some $v$, which holds if and only if $m_v \geq p$.

(3) Let $\mu \neq \mu'$ be distinct multisets, and put $k := |\mu|$, $k' := |\mu'|$. The sets $\mathcal{B}_\mu$ and $\mathcal{B}_{\mu'}$ consist of basis monomials whose underlying multisets are $\mu$ and $\mu'$ respectively, hence are disjoint. By (1), $\Sym_k(v_\mu)$ and $\Sym_{k'}(v_{\mu'})$ have disjoint supports in the standard basis $\mathcal{B}$, and under the hypothesis (2) ensures both are non-zero. Non-zero elements with pairwise disjoint supports in the standard basis are $\F$-linearly independent.
Therefore, the assertion follows.
\end{proof}

\begin{lemma}\label{lem:Vmacro-nonzero-indep}
For any field $\F$ and any directed tensor-labeled hypergraph $\Hbb$ following a single standard construction, $\Vmacro \setminus \{0\}$ is $\F$-linearly independent in $T(\F^{Q_0})$.
\end{lemma}

\begin{proof}
For each of the standard constructions (1)--(6), we show that the elements of $\Vmacro \setminus \{0\}$ form a family of non-zero tensors with pairwise disjoint supports in the standard basis $\mathcal{B}$. The claim then follows from the fact that non-zero elements with pairwise disjoint supports in the standard basis are $\F$-linearly independent. Since $\Vmacro$ is a set, multiple edges yielding the same tensor (for example through parallel hyperedges) appear only once in $\Vmacro$. In what follows we show that distinct non-zero tensors in $\Vmacro$ have pairwise disjoint supports.

For Construction (2), $\Vmacro$ is a subset of $Q_0$, which is the standard basis of $T^1(\F^{Q_0})$. Distinct vertices are distinct standard basis elements, so each element of $\Vmacro \setminus \{0\}$ is non-zero with singleton support, and these supports are pairwise disjoint.

For Construction (1), the source tensor of a non-loop edge $\{u, v\}$ ($u \neq v$) is $u \otimes v + v \otimes u \in T^2(\F^{Q_0})$, and the source tensor of a loop edge $\{v\}$ is $2(v \otimes v) \in T^2(\F^{Q_0})$. In characteristic $2$, the source tensor of a loop edge equals $0$ and is therefore excluded from $\Vmacro \setminus \{0\}$, although it may still belong to $\Vmacro$. The target tensor of every hyperedge is $1 \in T^0(\F^{Q_0})$. For distinct vertex pairs $\{u, v\} \neq \{u', v'\}$, the corresponding non-loop source tensors are distinct elements of $\Vmacro$ with supports $\{u \otimes v, v \otimes u\}$ and $\{u' \otimes v', v' \otimes u'\}$ that are disjoint. For distinct vertices $v \neq v'$, the loop source tensors (non-zero when $\mathrm{char}(\F) \neq 2$) have disjoint singleton supports $\{v \otimes v\}$ and $\{v' \otimes v'\}$. Between non-loop and loop source tensors, $u \neq v$ implies $u \otimes v, v \otimes u \neq v' \otimes v'$, so the supports are disjoint. Finally, the support $\{1\}$ of the target tensor lies in $T^0(\F^{Q_0})$, which is in a different degree from the supports above, hence disjoint.

For Constructions (3) and (5), each $A_e$ (and each $B_e$ in Construction (5)) is of the form $\Sym_k(v_\mu)$ for some multiset $\mu$. By Proposition \ref{prop:sym-vanishing}~(1), the support of $\Sym_k(v_\mu)$ in the standard basis is $\mathcal{B}_\mu$, and for distinct multisets $\mu \neq \mu'$ we have $\mathcal{B}_\mu \cap \mathcal{B}_{\mu'} = \emptyset$. Since only elements of $\Vmacro \setminus \{0\}$ are considered, Proposition \ref{prop:sym-vanishing}~(2) restricts attention to non-zero $\Sym_k(v_\mu)$. Hence distinct symmetric tensors in $\Vmacro \setminus \{0\}$ correspond to distinct multisets and have disjoint supports. In Construction (3), the target tensor of every hyperedge is $1 \in T^0(\F^{Q_0})$, which lies in a different degree from the source supports and is therefore disjoint.

For Constructions (4) and (6), each $A_e$ (and each $B_e$ in Construction (6)) is a pure tensor $u_1 \otimes \cdots \otimes u_k$ associated with an ordered tuple $(u_1, \ldots, u_k)$. This is itself an element of the standard basis $\mathcal{B}$ and is non-zero by construction. Distinct pure tensors in $\Vmacro$ are distinct standard basis elements with disjoint singleton supports. In Construction (4), the target tensor of every hyperedge is $1 \in T^0(\F^{Q_0})$, which lies in a different degree from the source supports and is therefore disjoint.

This establishes, for each standard construction, that $\Vmacro \setminus \{0\}$ is a family of non-zero tensors with pairwise disjoint supports in the standard basis.
\end{proof}

\begin{theorem}\label{thm:defect-vanishing}
Let $\F$ be a field and $\Hbb$ a directed tensor-labeled hypergraph following a single standard construction. Then
\begin{equation}\label{eq:delta-vanishes}
\delta(\Hbb) = 0.
\end{equation}
In particular, $\dim_{\F} \Zsp(\Hbb) = |Q_1| - |\Vmacro| + \cmacro$.
\end{theorem}

\begin{proof}
Take an arbitrary $\eta \in \Img(\Bmacro) \cap \Ker(\hphi) \subseteq \F^{\Vmacro}$, and show $\eta = 0$.

Expand $\eta$ in the standard basis:
$$\eta = \sum_{w \in \Vmacro} c_w \mathbf{1}_w \quad (c_w \in \F).$$
If $0 \in \Vmacro$, note that the corresponding basis vector $\mathbf{1}_0 \in \F^{\Vmacro}$ is distinct from the zero element $0 \in T(\F^{Q_0})$ in the codomain. To unify the notation, write $c_0$ for the coefficient of $\mathbf{1}_0$ when $0 \in \Vmacro$, and set $c_0 := 0$ when $0 \notin \Vmacro$.

By the definition of the evaluation map $\hphi$ and the assumption $\eta \in \Ker(\hphi)$,
$$0 = \hphi(\eta) = c_0 \cdot 0 + \sum_{w \in \Vmacro \setminus \{0\}} c_w \cdot w = \sum_{w \in \Vmacro \setminus \{0\}} c_w \cdot w.$$
By Lemma \ref{lem:Vmacro-nonzero-indep}, $\Vmacro \setminus \{0\}$ is $\F$-linearly independent in $T(\F^{Q_0})$, so $c_w = 0$ for every $w \in \Vmacro \setminus \{0\}$. Hence $\eta = c_0 \mathbf{1}_0$. If $0 \notin \Vmacro$, then $c_0 = 0$ and we already have $\eta = 0$.

It remains to show $c_0 = 0$ when $0 \in \Vmacro$. By Lemma \ref{lem:Bmacro-image}, any element of $\Img(\Bmacro)$ has zero coordinate sum on each weakly connected component of $\Hmacro$. Let $C_0 \subseteq \Vmacro$ denote the component containing the zero tensor $0$. The coordinate sum of $\eta = c_0 \mathbf{1}_0$ on $C_0$ is $c_0$. From $\eta \in \Img(\Bmacro)$ we obtain $c_0 = 0$, and so $\eta = 0$.

This proves $\delta(\Hbb) = 0$. The second assertion is immediate from Theorem \ref{thm:tensor-dim}.
\end{proof}

\begin{proposition}\label{prop:parallel-cycle}
Let $\Hbb$ be a directed tensor-labeled hypergraph with $k \geq 2$ distinct hyperedges $e_1, \ldots, e_k \in Q_1$ satisfying $\beta(e_1) = \cdots = \beta(e_k)$. For each $i = 2, \ldots, k$, set
$$\xi_i := \mathbf{1}_{e_1} - \mathbf{1}_{e_i} \in \F^{Q_1}.$$
Then $\xi_i \in \Zsp_{\mathrm{top}}(\Hbb) \subseteq \Zsp(\Hbb)$, and $\{\xi_2, \ldots, \xi_k\}$ is $\F$-linearly independent in $\F^{Q_1}$. In particular, $\Zsp_{\mathrm{top}}(\Hbb)$ contains the $(k-1)$-dimensional subspace $\Span_{\F}\{\xi_2, \ldots, \xi_k\}$.
\end{proposition}

\begin{proof}
For each $i \in \{2, \ldots, k\}$, the equality $\beta(e_1) = \beta(e_i)$ gives $A_{e_1} = A_{e_i}$ and $B_{e_1} = B_{e_i}$. Hence
\begin{align*}
&\partial_\beta(\xi_i) = (B_{e_1} - A_{e_1}) - (B_{e_i} - A_{e_i}) = 0, \\
& \Bmacro(\xi_i) = (\mathbf{1}_{B_{e_1}} - \mathbf{1}_{A_{e_1}}) - (\mathbf{1}_{B_{e_i}} - \mathbf{1}_{A_{e_i}}) = 0,
\end{align*}
so $\xi_i \in \Ker(\Bmacro) = \Zsp_{\mathrm{top}}(\Hbb) \subseteq \Ker(\partial_\beta) = \Zsp(\Hbb)$.

For linear independence, suppose $\sum_{i=2}^k a_i \xi_i = 0$ for some $a_i \in \F$. Expanding in the basis $\{\mathbf{1}_e \mid e \in Q_1\}$, the coefficient of $\mathbf{1}_{e_j}$ on the left side is $-a_j$ for each $j \in \{2, \ldots, k\}$. Hence $a_j = 0$ for $j = 2, \ldots, k$, and $\{\xi_2, \ldots, \xi_k\}$ is $\F$-linearly independent.
\end{proof}

\begin{remark}\label{rem:caveats}
Two caveats about Theorem \ref{thm:defect-vanishing}.

(1) The single-standard-construction hypothesis is essential. For instance, an ordered-tuple hyperedge $e_1 = (v, v)$ in Construction (4) yields $A_{e_1} = v \otimes v$, while a loop $e_2 = \{v\}$ in Construction (1) yields $A_{e_2} = 2(v \otimes v)$. If $\mathrm{char}(\F) \neq 2$, both are non-zero and $\Vmacro$ contains the linearly dependent pair $\{v \otimes v, 2(v \otimes v)\}$. The independence of Lemma \ref{lem:Vmacro-nonzero-indep} fails, and $\delta(\Hbb) > 0$ may occur under a suitable macrograph structure.

(2) In Construction (1) over $\mathrm{char}(\F) = 2$ with $k \geq 2$ parallel loops $\{e_0^{(1)}, \ldots, e_0^{(k)}\}$ at a single vertex $v$, each $\beta(e_0^{(i)}) = (0, 1)$. Proposition \ref{prop:parallel-cycle} gives a $(k-1)$-dimensional space of topological cycles, which arises from $\Zsp_{\mathrm{top}}(\Hbb)$ and does not contradict the theorem (which asserts $\delta = 0$ for the algebraic part). The effect of symmetric tensor vanishing appears in the topological, not the algebraic, cycle space.
\end{remark}

\begin{example}\label{ex:char2-loop}
Let $\F = \F_2$, and consider in Construction (1) an undirected graph with a loop $e_0$ at a vertex $v$. Since $\mathrm{char}(\F) = 2$, $\beta(e_0) = (2(v \otimes v), 1) = (0, 1)$, so $A_{e_0} = 0 \in T(\F^{Q_0})$ belongs to $\Vmacro$.
\begin{enumerate}
\item[(a)] Suppose $Q_1 = \{e_0\}$. Then $\Vmacro = \{0, 1\}$, and the macrograph is the single edge $0 \to 1$ with $\cmacro = 1$. Writing $\mathbf{1}_0, \mathbf{1}_1$ for the standard basis of $\F^{\Vmacro}$, we have $\Bmacro(\mathbf{1}_{e_0}) = \mathbf{1}_1 - \mathbf{1}_0$, so
$$\Img(\Bmacro) = \Span_{\F}\{\mathbf{1}_1 - \mathbf{1}_0\}.$$
Since $\hphi(\mathbf{1}_0) = 0$ and $\hphi(\mathbf{1}_1) = 1 \neq 0$,
$$\Ker(\hphi) = \Span_{\F}\{\mathbf{1}_0\}.$$
A common element $\alpha(\mathbf{1}_1 - \mathbf{1}_0) = \beta \mathbf{1}_0$ forces $\alpha = 0$ and $\beta = 0$. Hence $\Img(\Bmacro) \cap \Ker(\hphi) = 0$ and $\delta(\Hbb) = 0$.

\item[(b)] Suppose there are $k \geq 2$ parallel loops $e_0^{(1)}, \ldots, e_0^{(k)}$ at the same vertex $v$. All satisfy $\beta(e_0^{(i)}) = (0, 1)$, so $\Vmacro = \{0, 1\}$ is unchanged from (a), and Theorem \ref{thm:defect-vanishing} gives $\delta(\Hbb) = 0$. On the other hand, Proposition \ref{prop:parallel-cycle} applied to $e_0^{(1)}, \ldots, e_0^{(k)}$ yields $\xi_i := \mathbf{1}_{e_0^{(1)}} - \mathbf{1}_{e_0^{(i)}}$ ($i = 2, \ldots, k$), forming $k - 1$ linearly independent cycles in $\Zsp_{\mathrm{top}}(\Hbb)$.
\end{enumerate}
\end{example}

\subsection{Non-trivial algebraic cycles beyond the standard constructions}\label{subsec:nontrivial-defect}

Theorem \ref{thm:defect-vanishing} guarantees $\delta(\Hbb) = 0$ as long as $\Hbb$ follows a single standard construction. In this subsection we conversely show that allowing $\F$-linear combinations of vertex vectors as source or target tensors, beyond the standard constructions (1)--(6), can produce $\delta(\mathcal{H}) > 0$. We first give a general sufficient condition.

\begin{proposition}\label{prop:star-defect}
Let $\Hbb = (Q_0, Q_1, \beta)$ be a directed tensor-labeled hypergraph, and suppose there exist $r \geq 1$, distinct elements $w_0, w_1, \ldots, w_r \in \Vmacro$, and distinct hyperedges $e_1, \ldots, e_r \in Q_1$ such that $\beta(e_i) = (w_0, w_i)$ for $i = 1, \ldots, r$. If there exists a non-trivial $(\alpha_1, \ldots, \alpha_r) \in \F^r \setminus \{0\}$ with
$$\sum_{i=1}^r \alpha_i (w_i - w_0) = 0$$
in $T(\F^{Q_0})$, then $\xi := \sum_{i=1}^r \alpha_i \mathbf{1}_{e_i} \in \F^{Q_1}$ satisfies
$$\partial_\beta(\xi) = 0, \quad \Bmacro(\xi) \in \left(\Img(\Bmacro) \cap \Ker(\hphi)\right) \setminus \{0\}.$$
In particular, $\delta(\Hbb) \geq 1$.
\end{proposition}

\begin{proof}
By $\beta(e_i) = (w_0, w_i)$,
$$\partial_\beta(\xi) = \sum_{i=1}^r \alpha_i (w_i - w_0) = 0,$$
so $\xi \in \Zsp(\Hbb)$.

Next, we show $\Bmacro(\xi) \neq 0$. By the definition of $\Bmacro$, we obtain
$$\Bmacro(\xi) = \sum_{i=1}^r \alpha_i (\mathbf{1}_{w_i} - \mathbf{1}_{w_0}) = \sum_{i=1}^r \alpha_i \mathbf{1}_{w_i} - \left(\sum_{i=1}^r \alpha_i\right) \mathbf{1}_{w_0}.$$
Since $w_0, w_1, \ldots, w_r$ are distinct elements of $\Vmacro$, $\{\mathbf{1}_{w_0}, \mathbf{1}_{w_1}, \ldots, \mathbf{1}_{w_r}\}$ is linearly independent in $\F^{\Vmacro}$. As $(\alpha_1, \ldots, \alpha_r)$ is non-trivial, some $\alpha_i \neq 0$, and the coefficient of $\mathbf{1}_{w_i}$ in $\Bmacro(\xi)$ is non-zero. Hence $\Bmacro(\xi) \neq 0$.

By Proposition \ref{prop:tensor-factor} and $\partial_\beta(\xi) = 0$, we have $\Bmacro(\xi) \in \Ker(\hphi)$. Combining these, $\delta(\Hbb) = \dim_{\F}(\Img(\Bmacro) \cap \Ker(\hphi)) \geq 1$.
\end{proof}

\begin{example}\label{ex:alg-cycle}
Assume $\mathrm{char}(\F) \neq 2$. Consider the directed tensor-labeled hypergraph $\Hbb = (Q_0, Q_1, \beta)$ with $Q_0 = \{a, b\}$, $Q_1 = \{e_1, e_2\}$, and
$$\beta(e_1) = (a, a + b), \quad \beta(e_2) = (a, a - b).$$
Then $\Vmacro = \{a, a + b, a - b\}$, and the associated macrograph $\Hmacro$ is a star with $a$ as the centre and two edges, so $\cmacro = 1$. Since $W_{\Hbb} = \Span\{a, a + b, a - b\}$ has dimension $2$, $\dim_{\F} \Ker(\hphi) = 3 - 2 = 1$.

We have
$$\Bmacro(\mathbf{1}_{e_1} + \mathbf{1}_{e_2}) = \mathbf{1}_{a+b} + \mathbf{1}_{a-b} - 2\mathbf{1}_a \in \Img(\Bmacro) \cap \Ker(\hphi),$$
so $\delta(\Hbb) \geq 1$. On the other hand, $\Img(\Bmacro) \cap \Ker(\hphi) \subseteq \Ker(\hphi)$ and $\dim_{\F} \Ker(\hphi) = 1$, so $\delta(\Hbb) \leq 1$. Hence $\delta(\Hbb) = 1$.

This is the case $r = 2$, $w_0 = a$, $w_1 = a + b$, $w_2 = a - b$, $\alpha_1 = \alpha_2 = 1$ of Proposition \ref{prop:star-defect}.
\end{example}

The pair $(|Q_1|, |\Vmacro|) = (2, 3)$ in Example \ref{ex:alg-cycle} is minimal in both the number of edges and the number of macrograph vertices needed for $\delta(\Hbb) > 0$, as the following proposition shows.

\begin{proposition}\label{prop:minimality}
Any directed tensor-labeled hypergraph $\Hbb$ with $\delta(\Hbb) \geq 1$ satisfies $|Q_1| \geq 2$ and $|\Vmacro| \geq 3$.
\end{proposition}

\begin{proof}
By contraposition, we show $\delta(\Hbb) = 0$ whenever $|Q_1| \leq 1$ or $|\Vmacro| \leq 2$.

Case $|Q_1| \leq 1$: If $|Q_1| = 0$, then $\Vmacro = \emptyset$, $\Bmacro = 0$, and $\delta(\Hbb) = 0$ trivially. If $|Q_1| = 1$ with unique edge $e$, then $\Bmacro(\mathbf{1}_e) = \mathbf{1}_{B_e} - \mathbf{1}_{A_e}$. If $A_e = B_e$, then $\Img(\Bmacro) = 0$ and $\delta(\Hbb) = 0$. If $A_e \neq B_e$, then $\Img(\Bmacro) = \Span_{\F}\{\mathbf{1}_{B_e} - \mathbf{1}_{A_e}\}$ is one-dimensional, and any non-zero element $c(\mathbf{1}_{B_e} - \mathbf{1}_{A_e})$ ($c \in \F \setminus \{0\}$) is mapped by $\hphi$ to $c(B_e - A_e) \neq 0$, hence does not lie in $\Ker(\hphi)$. Thus $\Img(\Bmacro) \cap \Ker(\hphi) = 0$ and $\delta(\Hbb) = 0$.

Case $|Q_1| \geq 2$ and $|\Vmacro| \leq 2$: If $|\Vmacro| \leq 1$, then every hyperedge is a loop and $\Bmacro = 0$, so $\delta(\Hbb) = 0$. If $|\Vmacro| = 2$, write $\Vmacro = \{w_0, w_1\}$ with $w_0 \neq w_1$. For each $e \in Q_1$, $\Bmacro(\mathbf{1}_e) \in \{0, \pm(\mathbf{1}_{w_1} - \mathbf{1}_{w_0})\}$, so $\Img(\Bmacro) \subseteq \Span_{\F}\{\mathbf{1}_{w_1} - \mathbf{1}_{w_0}\}$ and $\dim_{\F} \Img(\Bmacro) \leq 1$. On the other hand, $\dim_{\F} \F^{\Vmacro} = 2$ and $\hphi(\F^{\Vmacro}) = \Span_{\F}\{w_0, w_1\}$. Since $w_0 \neq w_1$, at least one of $w_0, w_1$ is non-zero, so $\dim_{\F} \hphi(\F^{\Vmacro}) \geq 1$ and $\dim_{\F} \Ker(\hphi) \leq 1$.

For two subspaces of $\F^{\Vmacro}$ each of dimension at most $1$ to have non-zero intersection, both must be one-dimensional and equal. In that case $\mathbf{1}_{w_1} - \mathbf{1}_{w_0} \in \Ker(\hphi)$, that is, $w_1 - w_0 = 0$, contradicting $w_0 \neq w_1$. Hence $\Img(\Bmacro) \cap \Ker(\hphi) = 0$ and $\delta(\Hbb) = 0$.
\end{proof}

\section{Observation maps and projected cycle spaces}\label{sec:observation}

We construct a hierarchy of cycle spaces of varying granularity by composing $\partial_\beta$ with an \emph{observation map}---an $\F$-linear map sending higher-degree tensor components to lower-degree ones. With an appropriate choice, the classical cycle space is recovered from the cycle space of a higher-tensor encoding (Theorem \ref{thm:F2-recovery}).

By the \textbf{classical cycle space}, in the directed-graph case we mean the cycle space $\Zsp(D) = \Ker(B_D)$ of Section \ref{sec:general}, and in the undirected-graph case (Construction (1)) we mean the kernel $\Ker(B^{\mathrm{cl}}) \subseteq \F^{Q_1}$ of the classical undirected incidence matrix $B^{\mathrm{cl}} \in \{0, 1\}^{Q_0 \times Q_1}$ (defined by $B^{\mathrm{cl}}_{x, e} = 1$ if $x \in \psi(e)$ and $e$ is a non-loop edge, and $0$ otherwise), read over the appropriate field $\F$ \cite{Biggs, GodsilRoyle}.

\subsection{Definition of observation maps and projected cycle spaces}

\begin{definition}
Let $\Hbb$ be a directed tensor-labeled hypergraph and $\rho : T(\F^{Q_0}) \to U'$ an $\F$-linear map for some $\F$-vector space $U'$. We call $\rho$ an \textbf{observation map}, and define the corresponding \textbf{projected cycle space} by
$$\Zsp_{\rho}(\Hbb) := \Ker(\rho \circ \partial_{\beta}) \subseteq \F^{Q_1}.$$
\end{definition}

\begin{definition}\label{def:obs-defect}
For $\Hbb$ and an observation map $\rho$, the \textbf{defect invariant with respect to $\rho$} is
$$\delta_{\rho}(\Hbb) := \dim_{\F}\left(\Img(\Bmacro) \cap \Ker(\rho \circ \hphi)\right).$$
\end{definition}

\begin{theorem}\label{thm:obs-dim}
For $\Hbb$ and any observation map $\rho$,
\begin{equation}\label{eq:obs-dim}
\dim_{\F} \Zsp_{\rho}(\Hbb) = |Q_1| - |\Vmacro| + \cmacro + \delta_{\rho}(\Hbb).
\end{equation}
\end{theorem}

\begin{proof}
We have $\rho \circ \partial_{\beta} = (\rho \circ \hphi) \circ \Bmacro$, so applying Theorem \ref{thm:rank-nullity} to the labeling $w \mapsto \rho(w)$ (whose linear extension is $\rho \circ \hphi$) gives the claim.
\end{proof}

\begin{corollary}
If $\Ker(\rho_1 \circ \hphi) \subseteq \Ker(\rho_2 \circ \hphi)$, then $\Zsp_{\rho_1}(\Hbb) \subseteq \Zsp_{\rho_2}(\Hbb)$. In particular, for any observation map $\rho$,
$$\Zsp(\Hbb) \subseteq \Zsp_{\rho}(\Hbb) \subseteq \F^{Q_1}.$$
\end{corollary}

\begin{proposition}\label{prop:observation-quotient}
For any observation map $\rho : T(\F^{Q_0}) \to U'$, the restriction of $\partial_\beta$ yields the short exact sequence
$$0 \longrightarrow \Zsp(\Hbb) \longrightarrow \Zsp_\rho(\Hbb) \xrightarrow{\partial_\beta} \Img(\partial_\beta) \cap \Ker(\rho) \longrightarrow 0.$$
In particular, $\Zsp_\rho(\Hbb) / \Zsp(\Hbb) \cong \Img(\partial_\beta) \cap \Ker(\rho)$, and
$$\delta_\rho(\Hbb) - \delta(\Hbb) = \dim_\F \Zsp_\rho(\Hbb) - \dim_\F \Zsp(\Hbb) = \dim_\F\left(\Img(\partial_\beta) \cap \Ker(\rho)\right).$$
\end{proposition}

\begin{proof}
We have $\Zsp(\Hbb) = \Ker(\partial_\beta) \subseteq \Ker(\rho \circ \partial_\beta) = \Zsp_\rho(\Hbb)$. The restriction of $\partial_\beta$ to $\Zsp_\rho(\Hbb)$ sends $\xi \in \Zsp_\rho(\Hbb)$ to an element of $\Img(\partial_\beta) \cap \Ker(\rho)$, since $\rho(\partial_\beta(\xi)) = 0$. Its kernel is $\Zsp_\rho(\Hbb) \cap \Ker(\partial_\beta) = \Zsp(\Hbb)$. For surjectivity, any $w = \partial_\beta(\eta) \in \Img(\partial_\beta) \cap \Ker(\rho)$ comes from $\eta \in \Zsp_\rho(\Hbb)$, since $\rho(\partial_\beta(\eta)) = \rho(w) = 0$. This establishes the short exact sequence, and the isomorphism of quotients follows from the first isomorphism theorem. The final equality follows from Theorem \ref{thm:obs-dim} and Theorem \ref{thm:tensor-dim}.
\end{proof}

\subsection{Recovery of the classical cycle space via observation maps}\label{subsec:Z-vs-classical}

When $\Hbb$ follows Construction (2) (an ordinary directed graph), each $A_e, B_e \in Q_0 \subset T^1(\F^{Q_0})$, so $\partial_{\beta}(\mathbf{1}_e) = B_e - A_e \in T^1(\F^{Q_0})$ and $\Img(\partial_{\beta}) \subseteq T^1(\F^{Q_0})$ holds trivially. Hence for the projection $\pi_1 : T(\F^{Q_0}) \to T^1(\F^{Q_0}) = \F^{Q_0}$ onto degree $1$, $\pi_1 \circ \partial_{\beta} = \partial_{\beta}$ and $\Zsp_{\pi_1}(\Hbb) = \Zsp(\Hbb)$. Thus the classical directed-graph cycle space is recovered for Construction (2).

On the other hand, Construction (1) encodes an undirected edge $\{u, v\}$ as the symmetric $2$-tensor $u \otimes v + v \otimes u \in T^2(\F^{Q_0})$, so $\partial_{\beta}$ takes values in both $T^0(\F^{Q_0})$ and $T^2(\F^{Q_0})$, and $\Zsp(\Hbb)$ does not in general coincide with the classical undirected $\F_2$-coefficient cycle space or its oriented analogue.

\begin{example}\label{ex:triangle}
Let $Q_0 = \{a, b, c\}$, $Q_1 = \{e_{ab}, e_{bc}, e_{ca}\}$, and consider the triangle graph $K_3$ as a directed tensor-labeled hypergraph $\Hbb$ via Construction (1):
$$\beta(e_{ab}) = (a \otimes b + b \otimes a,\, 1), \quad \beta(e_{bc}) = (b \otimes c + c \otimes b,\, 1), \quad \beta(e_{ca}) = (c \otimes a + a \otimes c,\, 1).$$
The classical undirected cycle space $\Zsp_{\mathrm{cl}}(K_3)$ is one-dimensional over $\F_2$, generated by $\mathbf{1}_{e_{ab}} + \mathbf{1}_{e_{bc}} + \mathbf{1}_{e_{ca}}$. We denote this generator, read as an element of $\F^{Q_1}$, by $\xi$.

Applying the tensor incidence operator $\partial_\beta$ to $\xi$,
$$\partial_\beta(\xi) = 3 \cdot 1 - \left((a \otimes b + b \otimes a) + (b \otimes c + c \otimes b) + (c \otimes a + a \otimes c)\right).$$
The $T^0(\F^{Q_0})$ component is $3 \cdot 1$, and the $T^2(\F^{Q_0})$ component consists of the six standard basis words $a \otimes b$, $b \otimes a$, $b \otimes c$, $c \otimes b$, $c \otimes a$, $a \otimes c$, each with coefficient $-1$. Hence $\xi \notin \Zsp(\Hbb)$, so $\Zsp(\Hbb)$ does not coincide with the classical cycle space.

Moreover, $\partial_\beta : \F^{Q_1} \to T(\F^{Q_0})$ is injective, so $\Zsp(\Hbb) = 0$. Its dimension is $0$, which does not match the dimension $|Q_1| - |Q_0| + c(K_3) = 3 - 3 + 1 = 1$ of the classical undirected cycle space.

The result is also confirmed by Theorem \ref{thm:tensor-dim} and Theorem \ref{thm:defect-vanishing}. The set of boundary tensors is
$$\Vmacro = \{a \otimes b + b \otimes a,\; b \otimes c + c \otimes b,\; c \otimes a + a \otimes c,\; 1\},$$
of size $4$, and the associated macrograph $\Hmacro$ is a connected star with $1$ as the common target of the three edges, so $\cmacro = 1$. Following the single standard Construction (1), Theorem \ref{thm:defect-vanishing} gives $\delta(\Hbb) = 0$, and the dimension formula yields $\dim_{\F} \Zsp(\Hbb) = 3 - 4 + 1 + 0 = 0$.
\end{example}

\begin{theorem}\label{thm:F2-recovery}
Let $\F$ be a field of characteristic $2$, and let $\Hbb = (Q_0, Q_1, \beta)$ be an undirected graph encoded via Construction (1). Define the $\F$-linear map $\rho : T(\F^{Q_0}) \to \F^{Q_0}$ on the standard basis $\mathcal{B}$ \eqref{eq:standard-basis} by
\begin{equation}\label{eq:rho-recovery}
\rho(1) := 0, \quad \rho(u_1 \otimes u_2 \otimes \cdots \otimes u_k) := u_1 \quad (k \geq 1),
\end{equation}
and extend by $\F$-linearity. We call $\rho$ the \textbf{projection onto the first component}. Then $\Zsp_{\rho}(\Hbb)$ coincides with the kernel $\Ker(B^{\mathrm{cl}}) \subseteq \F^{Q_1}$ of the classical undirected incidence matrix $B^{\mathrm{cl}}$ over $\F$.
\end{theorem}

\begin{proof}
We compute $\rho(\partial_\beta(\mathbf{1}_e))$ for each edge of Construction (1).

For a non-loop edge $e = \{u, v\}$ ($u \neq v$), using $\mathrm{char}(\F) = 2$,
$$\rho(\partial_\beta(\mathbf{1}_e)) = 0 - (u + v) = -(u + v) = u + v.$$

For a loop edge $e = \{v\}$, $\beta(e) = (2(v \otimes v),\, 1) = (0, 1)$, so $\partial_\beta(\mathbf{1}_e) = 1$ and $\rho(\partial_\beta(\mathbf{1}_e)) = \rho(1) = 0$.

Hence, for any $\xi = \sum_{e \in Q_1} a_e \mathbf{1}_e \in \F^{Q_1}$,
\begin{equation}\label{eq:rho-circ-Delta}
(\rho \circ \partial_\beta)(\xi) = \sum_{\substack{e \in Q_1 \\ e = \{u_e, v_e\} \text{ non-loop}}} a_e (u_e + v_e) \in \F^{Q_0}.
\end{equation}
On the other hand, $B^{\mathrm{cl}}$ sends a non-loop edge $\{u, v\}$ to  $u + v$ and a loop edge to the zero vector. Its entries are $0$ or $1$, so $\rho \circ \partial_\beta$ and $B^{\mathrm{cl}}$ coincide on the basis $\{\mathbf{1}_e\}_{e \in Q_1}$, hence $\rho \circ \partial_\beta = B^{\mathrm{cl}}$ as $\F$-linear maps.
\end{proof}

\subsection{Filtration structure of observation maps}

Corresponding to the degree filtration of the tensor algebra $T(\F^{Q_0})$,
$$T^{\leq 0}(\F^{Q_0}) \subset T^{\leq 1}(\F^{Q_0}) \subset T^{\leq 2}(\F^{Q_0}) \subset \cdots, \quad T^{\leq k}(\F^{Q_0}) := \bigoplus_{j \leq k} T^j(\F^{Q_0}),$$
we define a degree filtration of observation maps.

\begin{definition}\label{def:rank-k-obs}
Let $\pi_{\leq k} : T(\F^{Q_0}) \to T^{\leq k}(\F^{Q_0})$ be the projection onto degrees $\leq k$. We call $\Zsp_{\leq k}(\Hbb) := \Zsp_{\pi_{\leq k}}(\Hbb)$ the \textbf{degree-$\leq k$ cycle space}.
\end{definition}

\begin{proposition}\label{prop:rank-filtration}
Let $K$ be the maximal tensor degree appearing in $\Hbb$. Then
$$\Zsp(\Hbb) = \Zsp_{\leq K}(\Hbb) \subseteq \Zsp_{\leq K-1}(\Hbb) \subseteq \cdots \subseteq \Zsp_{\leq 0}(\Hbb).$$
\end{proposition}

\begin{proof}
$\Ker(\pi_{\leq k} \circ \partial_{\beta}) \supseteq \Ker(\partial_{\beta})$, and from $\pi_{\leq k-1} = \pi_{\leq k-1} \circ \pi_{\leq k}$ we get $\Ker(\pi_{\leq k-1}) \supseteq \Ker(\pi_{\leq k})$.
\end{proof}

\begin{definition}\label{def:rank-k-defect}
For each $k \geq 0$, the \textbf{degree-$\leq k$ defect invariant} is
$$\delta_{\leq k}(\Hbb) := \dim_{\F}\left(\Img(\Bmacro) \cap \Ker(\pi_{\leq k} \circ \hphi)\right).$$
In particular, for $k \geq K$ (observing all tensor components), $\pi_{\leq k} \circ \hphi = \hphi$ and $\delta_{\leq k}(\Hbb) = \delta(\Hbb)$.
\end{definition}

\begin{theorem}\label{thm:rank-dim}
For each $k \geq 0$,
\begin{equation}\label{eq:rank-dim}
\dim_{\F} \Zsp_{\leq k}(\Hbb) = |Q_1| - |\Vmacro| + \cmacro + \delta_{\leq k}(\Hbb).
\end{equation}
\end{theorem}

\begin{proof}
Apply Theorem \ref{thm:obs-dim} with observation map $\rho := \pi_{\leq k}$. From $\rho \circ \partial_{\beta} = (\pi_{\leq k} \circ \hphi) \circ \Bmacro$, applying Theorem \ref{thm:rank-nullity} to the linear extension $\pi_{\leq k} \circ \hphi : \F^{\Vmacro} \to T^{\leq k}(\F^{Q_0})$ gives
$$\dim_{\F} \Ker(\rho \circ \partial_{\beta}) = |Q_1| - |\Vmacro| + \cmacro + \dim_{\F}\left(\Img(\Bmacro) \cap \Ker(\pi_{\leq k} \circ \hphi)\right).$$
The last term equals $\delta_{\leq k}(\Hbb)$ by Definition \ref{def:rank-k-defect}.
\end{proof}

\begin{corollary}\label{cor:filtration-chain}
If $k \leq k'$, then $\Ker(\pi_{\leq k} \circ \hphi) \supseteq \Ker(\pi_{\leq k'} \circ \hphi)$, so
$$\delta_{\leq 0}(\Hbb) \geq \delta_{\leq 1}(\Hbb) \geq \cdots \geq \delta_{\leq K}(\Hbb) = \delta(\Hbb).$$
\end{corollary}

\begin{proof}
From $\pi_{\leq k} = \pi_{\leq k} \circ \pi_{\leq k'}$ ($k \leq k'$), we have $\Ker(\pi_{\leq k} \circ \hphi) \supseteq \Ker(\pi_{\leq k'} \circ \hphi)$. The dimensions of intersections with $\Img(\Bmacro)$ are therefore monotonically non-increasing.
\end{proof}

The associated graded pieces of the filtration $\{\Zsp_{\leq k}(\Hbb)\}_k$ realize, at each tensor degree $k$, the space of cycles captured at degree $\leq k - 1$ but not at degree $\leq k$. We formulate this in the next proposition.

\begin{proposition}\label{prop:gr-realization}
Let $K$ be the maximal tensor degree appearing in $\Hbb$. For each $1 \leq k \leq K$, let $\pi_k : T(\F^{Q_0}) \to T^k(\F^{Q_0})$ be the projection onto degree $k$, and put $\partial_\beta^{(k)} := \pi_k \circ \partial_\beta : \F^{Q_1} \to T^k(\F^{Q_0})$. Then $\Zsp_{\leq k}(\Hbb) = \bigcap_{j=0}^{k} \Ker(\partial_\beta^{(j)})$, and the kernel of the restriction
$$\partial_\beta^{(k)}\big|_{\Zsp_{\leq k-1}(\Hbb)} : \Zsp_{\leq k-1}(\Hbb) \longrightarrow T^k(\F^{Q_0})$$
is $\Zsp_{\leq k}(\Hbb)$. In particular, $\Zsp_{\leq k-1}(\Hbb) / \Zsp_{\leq k}(\Hbb) \cong \partial_\beta^{(k)}(\Zsp_{\leq k-1}(\Hbb))$, and
$$\dim_\F\left(\Zsp_{\leq k-1}(\Hbb) / \Zsp_{\leq k}(\Hbb)\right) = \delta_{\leq k-1}(\Hbb) - \delta_{\leq k}(\Hbb).$$
\end{proposition}

\begin{proof}
For any $\xi \in \F^{Q_1}$, $\pi_{\leq k}(\partial_\beta(\xi)) = 0$ if and only if all degree-$0, 1, \ldots, k$ components vanish, so $\Zsp_{\leq k}(\Hbb) = \Ker(\pi_{\leq k} \circ \partial_\beta) = \bigcap_{j=0}^{k} \Ker(\partial_\beta^{(j)})$. Consequently $\Zsp_{\leq k}(\Hbb) = \Zsp_{\leq k-1}(\Hbb) \cap \Ker(\partial_\beta^{(k)})$, and the kernel of the restriction of $\partial_\beta^{(k)}$ to $\Zsp_{\leq k-1}(\Hbb)$ is precisely $\Zsp_{\leq k}(\Hbb)$. The isomorphism follows from the first isomorphism theorem, and the dimension equality from Theorem \ref{thm:rank-dim}.
\end{proof}

\section{The edge Gram operator of the tensor incidence operator}\label{sec:spectrum}

We introduce an inner product structure over $\F = \R$ and study the edge-edge symmetric positive semi-definite operator $\Lap := \partial_{\beta}^{*} \partial_{\beta}$, which is the Gram matrix of the family $\{\partial_{\beta}(\mathbf{1}_e)\}$ of edge difference vectors. The operator $\Lap$ is formally analogous to the classical graph Laplacian \cite{Chung, ChungSpectral}. In the hypergraph setting, alternative spectral notions have been studied via tensor eigenvalues \cite{CooperDutle} and via second-eigenvalue estimates \cite{FriedmanWigderson}. Whether the spectrum of $\Lap$ reflects the adjacency structure depends strongly on the construction: in Construction (2) it coincides with the classical edge Laplacian of a directed graph (Proposition \ref{prop:dir-graph-Lap}), while for a simple graph in Construction (1) it depends only on the number of edges (Proposition \ref{prop:loopless-simple-Lap}). We treat $\Lap$ here as an edge operator of Gram type, and give combinatorial descriptions of its kernel ($= \Zsp(\Hbb)$) and rank.

\subsection{The edge Gram operator \texorpdfstring{$\Lap$}{L\_beta}}

Throughout this section we work over $\F = \R$. The same results hold over $\C$ using the Hermitian inner product with the standard basis as an orthonormal basis and defining the adjoint as the conjugate transpose (for example, $\Lap = B_D^{\top} B_D$ in Proposition \ref{prop:dir-graph-Lap} becomes $\Lap = B_D^{*} B_D$ over $\C$). We restrict to the real case for the simplicity of notation. All defect invariants in this section are computed over $\R$. We write $\delta_{\R}(\Hbb)$ when explicit reference to the field is required, and continue to write $\delta(\Hbb)$ when no ambiguity arises.

Equip $\R^{Q_0}$ with the standard inner product having $\{\mathbf{1}_v\}_{v \in Q_0}$ as an orthonormal basis, and equip $\R^{Q_1}$ with the standard inner product having $\{\mathbf{1}_e\}_{e \in Q_1}$ as an orthonormal basis. Under the identification of Remark \ref{rem:word-space}, equip the tensor algebra $T(\R^{Q_0})$ with the inner product $\langle \cdot, \cdot \rangle$ having the standard basis $\mathcal{B}$ \eqref{eq:standard-basis} as an orthonormal basis. Explicitly, on pure tensors,
\begin{equation}\label{eq:tensor-inner-product}
\langle u_1 \otimes \cdots \otimes u_k, w_1 \otimes \cdots \otimes w_l \rangle = \delta_{kl} \prod_{i=1}^{k} \langle u_i, w_i \rangle_{\R^{Q_0}},
\end{equation}
where $\delta_{kl}$ is the Kronecker delta and pure tensors of different degrees are orthogonal.

\begin{proposition}\label{prop:adjoint}
The adjoint $\partial_{\beta}^{*} : T(\R^{Q_0}) \to \R^{Q_1}$ of the tensor incidence operator $\partial_{\beta} : \R^{Q_1} \to T(\R^{Q_0})$ is given by
\begin{equation}\label{eq:adjoint}
\partial_{\beta}^{*}(w) = \sum_{e \in Q_1} \langle B_e - A_e, w \rangle \cdot \mathbf{1}_e \quad (w \in T(\R^{Q_0})).
\end{equation}
In particular, $\Img(\partial_{\beta}^{*}) = \Ker(\partial_{\beta})^{\perp} = \Zsp(\Hbb)^{\perp}$.
\end{proposition}

\begin{proof}
For any $\xi = \sum_e a_e \mathbf{1}_e \in \R^{Q_1}$ and $w \in T(\R^{Q_0})$,
\begin{align*}
\langle \partial_{\beta}(\xi), w \rangle &= \left\langle \sum_e a_e (B_e - A_e), w \right\rangle = \sum_e a_e \langle B_e - A_e, w \rangle \\
&= \left\langle \xi, \sum_e \langle B_e - A_e, w \rangle \mathbf{1}_e \right\rangle_{\R^{Q_1}}.
\end{align*}
Since this holds for every $\xi$, the definition of the adjoint yields \eqref{eq:adjoint}.
\end{proof}

\begin{definition}\label{def:Laplacian}
The \textbf{edge Gram operator}\footnote{This is precisely the Gram matrix of the family $\{\partial_{\beta}(\mathbf{1}_e)\}_{e \in Q_1}$ of edge difference vectors. Since $\partial_{\beta}$ corresponds to the classical incidence matrix, $\Lap$ is formally analogous to the classical graph Laplacian, but its spectrum does not in general reflect the adjacency structure of the graph (Proposition \ref{prop:loopless-simple-Lap}).}  $\Lap : \R^{Q_1} \to \R^{Q_1}$ of $\Hbb = (Q_0, Q_1, \beta)$ is defined by
\begin{equation}\label{eq:Laplacian-def}
\Lap := \partial_{\beta}^{*} \partial_{\beta}.
\end{equation}
\end{definition}

\begin{proposition}\label{prop:Lap-structure}
$\Lap$ is a symmetric positive semi-definite operator on $\R^{Q_1}$, and the matrix entries in the standard basis $\{\mathbf{1}_e\}_{e \in Q_1}$ are
\begin{equation}\label{eq:Laplacian-matrix}
(\Lap)_{e, e'} = \langle B_e - A_e, B_{e'} - A_{e'} \rangle \quad (e, e' \in Q_1).
\end{equation}
In particular, all eigenvalues of $\Lap$ are non-negative real numbers.
\end{proposition}

\begin{proof}
The matrix entry in the standard basis is
$$\langle \Lap \mathbf{1}_{e'}, \mathbf{1}_e \rangle = \langle \partial_{\beta} \mathbf{1}_{e'}, \partial_{\beta} \mathbf{1}_e \rangle = \langle B_{e'} - A_{e'}, B_e - A_e \rangle,$$
which by the symmetry of the inner product coincides with the right side of \eqref{eq:Laplacian-matrix}. Thus $\Lap$ is represented by a symmetric matrix. Positive semi-definiteness follows from
$$\langle \Lap \xi, \xi \rangle = \langle \partial_{\beta}^{*} \partial_{\beta} \xi, \xi \rangle = \langle \partial_{\beta} \xi, \partial_{\beta} \xi \rangle = \|\partial_{\beta} \xi\|^2 \geq 0$$
for any $\xi \in \R^{Q_1}$.
\end{proof}

\begin{proposition}\label{prop:kernel-Lap}
The kernel of the edge Gram operator $\Lap$ coincides with the tensor cycle space $\Zsp(\Hbb)$ of $\Hbb$:
$$\Ker(\Lap) = \Zsp(\Hbb).$$
\end{proposition}

\begin{proof}
The inclusion $\Zsp(\Hbb) \subseteq \Ker(\Lap)$ is immediate. For the reverse inclusion, let $\xi \in \Ker(\Lap)$. Then $\Lap(\xi) = 0$ gives $\langle \Lap(\xi), \xi \rangle = 0$. By the definition of the adjoint,
$$\langle \Lap(\xi), \xi \rangle = \langle \partial_\beta^* \partial_\beta(\xi), \xi \rangle = \langle \partial_\beta(\xi), \partial_\beta(\xi) \rangle = \|\partial_\beta(\xi)\|^2,$$
so $\|\partial_\beta(\xi)\|^2 = 0$. By the positive definiteness of the inner product, $\partial_\beta(\xi) = 0$, that is, $\xi \in \Ker(\partial_\beta) = \Zsp(\Hbb)$.
\end{proof}

\begin{theorem}\label{thm:spectrum}
The rank of the edge Gram operator $\Lap$ is expressed in terms of  $|\Vmacro|$,  $\cmacro$, and $\delta(\Hbb)$ as
\begin{equation}\label{eq:spectrum-formula}
\rank(\Lap) = |Q_1| - \dim_{\R} \Zsp(\Hbb) = |\Vmacro| - \cmacro - \delta(\Hbb).
\end{equation}
In particular, the sum of the multiplicities of the non-zero eigenvalues of $\Lap$ also equals this value.
\end{theorem}

\begin{proof}
Since $\Lap$ is symmetric positive semi-definite, it is diagonalizable and $\R^{Q_1}$ decomposes as an orthogonal direct sum of eigenspaces. The eigenspace at $0$ is $\Ker(\Lap)$, which coincides with $\Zsp(\Hbb)$ by Proposition \ref{prop:kernel-Lap}, so the multiplicity of $0$ is $\dim_{\R} \Zsp(\Hbb)$. The complementary subspace of dimension $|Q_1| - \dim_{\R} \Zsp(\Hbb)$ is the orthogonal direct sum of the non-zero eigenspaces. Its dimension is $\rank(\Lap)$ and equals the total multiplicity of the non-zero eigenvalues. By Theorem \ref{thm:tensor-dim},
$$\dim_{\R} \Zsp(\Hbb) = |Q_1| - |\Vmacro| + \cmacro + \delta(\Hbb),$$
and substituting this gives \eqref{eq:spectrum-formula}.
\end{proof}

\begin{remark}\label{rem:rank-interpretation}
The quantities $|\Vmacro|$ and $\cmacro$ on the right side of Theorem \ref{thm:spectrum} are purely combinatorial invariants of the associated macrograph $\Hmacro$, whereas $\delta(\Hbb) = \dim_{\R}(\Img(\Bmacro) \cap \Ker(\hphi))$ measures the linear dependencies among the tensor labels via the evaluation map $\hphi$. In the general framework of Definition \ref{def:tensor-hg}, $\delta_\F(\Hbb)$ may depend on the base field or characteristic, so the rank formula \eqref{eq:spectrum-formula} reduces to purely combinatorial invariants only when $\delta_\R(\Hbb) = 0$. For $\Hbb$ following a single standard construction, Theorem \ref{thm:defect-vanishing} guarantees $\delta_\F(\Hbb) = 0$ for every field $\F$, and in that case $\rank(\Lap) = |\Vmacro| - \cmacro$.
\end{remark}

\begin{corollary}\label{cor:spectral-bounds}
Let $\lambda_{\max}$ and $\lambda_{\min}^{+}$ be the largest eigenvalue of $\Lap$  and the smallest non-zero eigenvalue of $\Lap$ , respectively. Let $P_{\Zsp}$ be the orthogonal projection onto $\Zsp(\Hbb)$. Then for any $\xi \in \R^{Q_1}$,
\begin{equation}\label{eq:spectral-bounds}
\lambda_{\min}^{+} \|\xi - P_{\Zsp} \xi\|^2 \leq \|\partial_{\beta} \xi\|^2 \leq \lambda_{\max} \|\xi - P_{\Zsp} \xi\|^2.
\end{equation}
\end{corollary}

\begin{proof}
The diagonalizability of $\Lap$ gives the orthogonal direct sum $\R^{Q_1} = \Ker(\Lap) \oplus \Ker(\Lap)^{\perp}$. By Proposition \ref{prop:kernel-Lap}, $\Ker(\Lap) = \Zsp(\Hbb)$, so $\xi$ decomposes as $\xi = P_{\Zsp} \xi + \xi^{\perp}$ with $\xi^{\perp} := \xi - P_{\Zsp} \xi \in \Zsp(\Hbb)^{\perp}$. Since $\Lap(P_{\Zsp} \xi) = 0$,
$$\|\partial_{\beta} \xi\|^2 = \langle \Lap \xi, \xi \rangle = \langle \Lap \xi^{\perp}, \xi^{\perp} \rangle.$$
The eigenvalues of $\Lap$ restricted to $\Zsp(\Hbb)^{\perp}$ are positive real numbers in the interval $[\lambda_{\min}^{+}, \lambda_{\max}]$, so
$$\lambda_{\min}^{+} \|\xi^{\perp}\|^2 \leq \langle \Lap \xi^{\perp}, \xi^{\perp} \rangle \leq \lambda_{\max} \|\xi^{\perp}\|^2,$$
which yields \eqref{eq:spectral-bounds}.
\end{proof}

\subsection{Observation Gram operators and spectral filtration}\label{subsec:obs-Gram}

Corresponding to the degree filtration $\{\Zsp_{\leq k}(\Hbb)\}_k$ introduced in Section \ref{sec:observation}, we construct a filtration of edge Gram operators.

\begin{definition}\label{def:obs-Lap}
Let $\pi_{\leq k} : T(\R^{Q_0}) \to T^{\leq k}(\R^{Q_0})$ be the orthogonal projection onto degrees $\leq k$, and $\pi_k : T(\R^{Q_0}) \to T^k(\R^{Q_0})$ the orthogonal projection onto degree $k$. The compositions
$$\partial_{\leq k} := \pi_{\leq k} \circ \partial_{\beta} : \R^{Q_1} \to T^{\leq k}(\R^{Q_0}), \quad \partial_{\beta}^{(k)} := \pi_k \circ \partial_{\beta} : \R^{Q_1} \to T^k(\R^{Q_0})$$
are called the \textbf{degree-$\leq k$ observed tensor incidence operator} and the \textbf{degree-$k$ component incidence operator}, respectively. The compositions with their adjoints,
$$L_{\leq k} := \partial_{\leq k}^{*} \partial_{\leq k} : \R^{Q_1} \to \R^{Q_1}, \quad L^{(k)} := (\partial_{\beta}^{(k)})^{*} \partial_{\beta}^{(k)} : \R^{Q_1} \to \R^{Q_1},$$
are called the \textbf{degree-$\leq k$ observation Gram operator} and the \textbf{degree-$k$ component Gram operator}, respectively.
\end{definition}

From the self-adjointness $\pi_{\leq k}^{*} = \pi_{\leq k}$ and $\pi_k^{*} = \pi_k$ of the orthogonal projections, together with the orthogonality of distinct degree components of $T(\R^{Q_0})$, for any $\xi \in \R^{Q_1}$,
\begin{equation}\label{eq:obs-quad}
\langle L_{\leq k} \xi, \xi \rangle = \|\pi_{\leq k} \partial_{\beta}(\xi)\|^2 = \sum_{j=0}^{k} \|\partial_{\beta}^{(j)}(\xi)\|^2.
\end{equation}

\begin{proposition}\label{prop:obs-Lap-sum}
Let $K$ be the maximal tensor degree appearing in $\Hbb$. For each $k \geq 0$,
$$L_{\leq k} = \sum_{j=0}^{k} L^{(j)},$$
and in particular $\Lap = L_{\leq K} = \sum_{j=0}^{K} L^{(j)}$. Consequently, for any $\xi \in \R^{Q_1}$,
\begin{equation}\label{eq:norm-decomposition}
\|\partial_{\beta}(\xi)\|^2 = \sum_{j=0}^{K} \|\partial_{\beta}^{(j)}(\xi)\|^2,
\end{equation}
which is an orthogonal decomposition of the squared norm of $\partial_{\beta}(\xi)$ by degree, equivalently a degree-wise orthogonal decomposition of the quadratic form $\xi \mapsto \langle \Lap \xi, \xi \rangle$.
\end{proposition}

\begin{proof}
By the orthogonality of distinct degree componants, $\pi_{\leq k} = \sum_{j=0}^{k} \pi_j$, where each $\pi_j$ is self-adjoint and satisfies $\pi_j \pi_{j'} = \delta_{jj'} \pi_j$. Hence
$$L_{\leq k} = \partial_{\beta}^{*} \pi_{\leq k} \partial_{\beta} = \partial_{\beta}^{*} \left(\sum_{j=0}^{k} \pi_j\right) \partial_{\beta} = \sum_{j=0}^{k} \partial_{\beta}^{*} \pi_j \partial_{\beta} = \sum_{j=0}^{k} L^{(j)}.$$
The equality $\Lap = L_{\leq K}$ follows from $\partial_{\beta}(\R^{Q_1}) \subseteq T^{\leq K}(\R^{Q_0})$, giving $\pi_{\leq K} \partial_{\beta} = \partial_{\beta}$. The orthogonal decomposition \eqref{eq:norm-decomposition} is obtained by setting $k = K$ in \eqref{eq:obs-quad}.
\end{proof}

\begin{theorem}\label{thm:obs-Lap-spectrum}
The degree-$\leq k$ observation Gram operator $L_{\leq k}$ satisfies the following.
\begin{enumerate}
\item[\textup{(1)}] $L_{\leq k}$ is a symmetric positive semi-definite operator on $\R^{Q_1}$.
\item[\textup{(2)}] $\Ker(L_{\leq k}) = \Zsp_{\leq k}(\Hbb)$.
\item[\textup{(3)}] $\rank(L_{\leq k}) = |\Vmacro| - \cmacro - \delta_{\leq k}(\Hbb)$.
\item[\textup{(4)}] If $k \leq \ell$, then $L_{\leq k} \preceq L_{\leq \ell}$ in the Loewner order. Writing the eigenvalues of $L_{\leq k}$ in ascending order as $0 \leq \lambda_1^{(k)} \leq \cdots \leq \lambda_{|Q_1|}^{(k)}$, we have $\lambda_i^{(k)} \leq \lambda_i^{(\ell)}$ for each $i$.
\end{enumerate}
\end{theorem}

\begin{proof}
(1) Symmetry is immediate from the form $L_{\leq k} = \partial_{\leq k}^{*} \partial_{\leq k}$, and positive semi-definiteness is clear from \eqref{eq:obs-quad}.

(2) Since $L_{\leq k}$ is positive semi-definite, $\xi \in \Ker(L_{\leq k})$ if and only if $\langle L_{\leq k} \xi, \xi \rangle = 0$. By \eqref{eq:obs-quad} this is equivalent to $\pi_{\leq k} \partial_{\beta}(\xi) = 0$, that is, $\xi \in \Ker(\pi_{\leq k} \circ \partial_{\beta}) = \Zsp_{\leq k}(\Hbb)$.

(3) By (2), $\rank(L_{\leq k}) = |Q_1| - \dim \Zsp_{\leq k}(\Hbb)$, and combining with Theorem \ref{thm:rank-dim} gives the formula.

(4) For $k \leq \ell$, applying \eqref{eq:obs-quad} to the difference $L_{\leq \ell} - L_{\leq k}$, for any $\xi \in \R^{Q_1}$,
$$\langle (L_{\leq \ell} - L_{\leq k}) \xi, \xi \rangle = \sum_{j=k+1}^{\ell} \|\partial_{\beta}^{(j)}(\xi)\|^2 \geq 0.$$
Hence $L_{\leq \ell} - L_{\leq k}$ is positive semi-definite, that is, $L_{\leq k} \preceq L_{\leq \ell}$ in the Loewner order. The monotonicity of the eigenvalues follows from the Courant-Fischer min-max principle.
\end{proof}

\begin{corollary}\label{cor:rank-increment}
Under the conventions $\delta_{\leq -1}(\Hbb) := |\Vmacro| - \cmacro$ and $L_{\leq -1} := 0$, for each $k \geq 0$, the following equation holds.
$$\rank(L_{\leq k}) - \rank(L_{\leq k-1}) = \delta_{\leq k-1}(\Hbb) - \delta_{\leq k}(\Hbb)$$
\end{corollary}

\begin{proof}
By Theorem \ref{thm:obs-Lap-spectrum}~(3), $\rank(L_{\leq k}) = |\Vmacro| - \cmacro - \delta_{\leq k}(\Hbb)$, and under the convention $k = -1$ this also gives $\rank(L_{\leq -1}) = 0$. The equation now follows from Proposition \ref{prop:gr-realization}.
\end{proof}

\subsection{Computations under standard constructions}

\subsubsection{Reduction to the classical edge Laplacian in Construction (2)}

\begin{proposition}\label{prop:dir-graph-Lap}
If $\Hbb$ follows Construction (2), then the tensor incidence operator $\partial_{\beta}$ coincides with the classical incidence matrix $B_D$, and
$$\Lap = B_D^{\top} B_D.$$
That is, $\Lap$ coincides with the signed edge Laplacian of the directed graph.
\end{proposition}

\begin{proof}
In Construction (2), $A_e = s(e)$ and $B_e = t(e) \in Q_0 \subset T^1(\R^{Q_0})$ for each edge $e$. The inner product \eqref{eq:tensor-inner-product} restricted to $T^1(\R^{Q_0})$ is the standard inner product of $\R^{Q_0}$, so $\partial_{\beta}(\mathbf{1}_e) = t(e) - s(e) = B_D(\mathbf{1}_e)$ as $\R$-linear maps. The adjoint $\partial_{\beta}^{*} = B_D^{\top}$ is the transpose with respect to the standard inner product, and $\Lap = \partial_{\beta}^{*} \partial_{\beta} = B_D^{\top} B_D$.
\end{proof}

\subsubsection{\texorpdfstring{$\Lap$}{L\_beta} for loopless simple graphs in Construction (1)}

\begin{proposition}\label{prop:loopless-simple-Lap}
Let $\Hbb$ be an undirected graph encoded via Construction (1) that contains neither loop edges nor parallel edges. Set $m := |Q_1|$. Then
\begin{equation}\label{eq:Lap-2I-J}
\Lap = 2 I_m + J_m,
\end{equation}
where $I_m$ is the $m \times m$ identity matrix and $J_m$ is the $m \times m$ matrix with all entries equal to $1$.
\end{proposition}

\begin{proof}
In Construction (1), each non-loop edge $e = \{u, v\}$ ($u \neq v$) has $A_e = u \otimes v + v \otimes u \in T^2(\R^{Q_0})$ and $B_e = 1 \in T^0(\R^{Q_0})$, so $\partial_{\beta}(\mathbf{1}_e) = 1 - (u \otimes v + v \otimes u)$. By \eqref{eq:tensor-inner-product}, the components in $T^0(\R^{Q_0})$ and $T^2(\R^{Q_0})$ are orthogonal, and the standard basis $\{u \otimes v\}_{u, v \in Q_0}$ is orthonormal in $T^2(\R^{Q_0})$. Hence the diagonal entry is
$$(\Lap)_{e,e} = \|1\|^2 + \|u \otimes v + v \otimes u\|^2 = 1 + 2 = 3.$$
By the absence of parallel edges, distinct edges $e \neq e'$ correspond to distinct vertex pairs $\{u, v\} \neq \{u', v'\}$, so the inner product expansion
$$\langle u \otimes v + v \otimes u,\, u' \otimes v' + v' \otimes u' \rangle = \delta_{u,u'} \delta_{v,v'} + \delta_{u,v'} \delta_{v,u'} + \delta_{v,u'} \delta_{u,v'} + \delta_{v,v'} \delta_{u,u'}$$
vanishes term by term, giving $\langle A_e, A_{e'} \rangle = 0$. Hence the off-diagonal entry is
$$(\Lap)_{e,e'} = \langle 1, 1 \rangle - \langle 1, A_{e'} \rangle - \langle A_e, 1 \rangle + \langle A_e, A_{e'} \rangle = 1 - 0 - 0 + 0 = 1,$$
and combining diagonal entries $3$ with off-diagonal entries $1$ yields $\Lap = 2 I_m + J_m$.
\end{proof}

\begin{remark}\label{rem:Lap-rigidity}
Proposition \ref{prop:loopless-simple-Lap} asserts that for a loopless simple graph in Construction (1), $\Lap$ depends only on the number of edges $m$, hence its spectrum reflects no adjacency structure. The rigidity persists at the level of the observation filtration of Section~\ref{subsec:obs-Gram}: $L^{(0)} = J_m$ and $L^{(2)} = 2 I_m$ likewise depend only on $m$. This is a fundamental difference from the classical graph Laplacian and justifies treating $\Lap$ as an edge operator of Gram type rather than a ``Laplacian''. Non-trivial structural information requires parallel edges, loops, symmetric tensor degenerations, or an encoding such as Construction~(2) (Proposition~\ref{prop:dir-graph-Lap}, Theorem~\ref{thm:F2-recovery}).
\end{remark}

\begin{corollary}\label{cor:Kn-spectrum}
For the complete undirected graph $K_n$ ($n \geq 2$) on $n$ vertices encoded via Construction (1) with $m := \binom{n}{2}$, $\Lap = 2 I_m + J_m$, and the eigenvalues are $m + 2$ with multiplicity $1$ (eigenvector $(1, 1, \ldots, 1)^{\top}$) and $2$ with multiplicity $m - 1$ (eigenspace $\{(c_1, \ldots, c_m) \in \R^m : \sum_e c_e = 0\}$). In particular, $\Lap$ has full rank and $\Zsp(K_n) = 0$.
\end{corollary}

\begin{proof}
$K_n$ is a loopless simple graph, so Proposition \ref{prop:loopless-simple-Lap} gives $\Lap = 2 I_m + J_m$. The eigenvalues of $J_m$ are $m$ with multiplicity $1$ (eigenvector $(1, \ldots, 1)^{\top}$) and $0$ with multiplicity $m - 1$ (eigenspace $\sum_e c_e = 0$), so the eigenvalues of $\Lap = 2 I_m + J_m$ are $m + 2$ and $2$. Since $\det(\Lap) = (m + 2) \cdot 2^{m-1} > 0$, $\Lap$ is invertible, $\Ker(\Lap) = 0$, and by Proposition \ref{prop:kernel-Lap}, $\Zsp(K_n) = 0$.
\end{proof}

\begin{remark}\label{rem:Z-not-classical}
The cycle space $\Zsp(\Hbb)$ in Construction (1) does not in general coincide with the classical undirected cycle space. For example, the classical cycle space of $K_n$ ($n \geq 3$) has dimension $\binom{n-1}{2}$, whereas Corollary \ref{cor:Kn-spectrum} gives $\Zsp(K_n) = \Ker(\Lap) = 0$. Recovery of the classical cycle space requires the observation map $\rho$ of Theorem \ref{thm:F2-recovery}.
\end{remark}

\begin{example}\label{ex:Lap-simple-graphs}
We compute $\Lap$ explicitly for typical loopless simple graphs.

For $K_3$, $m = 3$ and
$$\Lap = 2 I_3 + J_3 = \begin{pmatrix} 3 & 1 & 1 \\ 1 & 3 & 1 \\ 1 & 1 & 3 \end{pmatrix},$$
with eigenvalues $5$ of multiplicity $1$ (eigenvector $(1, 1, 1)^{\top}$) and $2$ of multiplicity $2$. We have $\det(\Lap) = 5 \cdot 4 = 20 \neq 0$.

For the path $P_n$ ($n \geq 2$), the number of edges is $m = n - 1$, so $\Lap = 2 I_{n-1} + J_{n-1}$, with eigenvalues $n + 1$ of multiplicity $1$ and $2$ of multiplicity $n - 2$. The set of boundary tensors is $\Vmacro = \{1\} \cup \{A_e\}_{e \in Q_1}$ of size $|\Vmacro| = n$, the associated macrograph is a star centred at $1$ with $\cmacro = 1$, and the single standard Construction (1) gives $\delta(P_n) = 0$ by Theorem \ref{thm:defect-vanishing}. Theorem \ref{thm:spectrum} yields $\rank(\Lap) = n - 1 - 0 = m$, consistent with the direct computation.

For the cycle $C_n$ ($n \geq 3$), the number of edges is $m = n$, so $\Lap = 2 I_n + J_n$, with eigenvalues $n + 2$ of multiplicity $1$ and $2$ of multiplicity $n - 1$. We have $|\Vmacro| = n + 1$, $\cmacro = 1$, $\delta(C_n) = 0$, hence $\rank(\Lap) = n = m$.

In each case, $\Lap$ has full rank and $\Zsp = 0$. The classical cycle space ($0$-dimensional for $P_n$, $1$-dimensional for $C_n$) does not appear directly in $\Zsp(\Hbb)$. Its recovery requires the observation map of Theorem \ref{thm:F2-recovery}.
\end{example}

\section{Correspondence with oriented hypergraph theory}\label{sec:related}

Oriented hypergraph theory \cite{ReffRusnak, Rusnak, ChenLiuRobinsonRusnakWang}, which has its precursors in the theory of signed graphs \cite{Zaslavsky} and bidirected graphs \cite{Bouchet}, describes vertex--edge incidences via signed scalar entries $\{-1, 0, +1\}$, and forms a parallel independent theory to the framework of this paper. In this section we construct a natural correspondence (a map) between the two, and make explicit the relation between their cycle spaces.

\begin{definition}\label{def:oriented-hg}
An \textbf{oriented hypergraph} $\Hbb_o = (Q_0, Q_1, \mathbb{B}^{\mathrm{oh}})$ is a triple consisting of a finite set of vertices $Q_0$, a finite set of hyperedges $Q_1$, and an incidence matrix $\mathbb{B}^{\mathrm{oh}} \in \{-1, 0, +1\}^{Q_0 \times Q_1}$. The entry $\mathbb{B}^{\mathrm{oh}}_{v, e} \in \{-1, 0, +1\}$ indicates whether the vertex $v$ is on the ``source side'' ($-1$), the ``target side'' ($+1$), or ``non-incident'' ($0$) for the hyperedge $e$.
\end{definition}

\begin{definition}\label{def:natural-F}
Given an oriented hypergraph $\Hbb_o = (Q_0, Q_1, \mathbb{B}^{\mathrm{oh}})$, we construct a directed tensor-labeled hypergraph $F(\Hbb_o) = (Q_0, Q_1, \beta_F)$ by
\begin{equation}\label{eq:natural-F}
\beta_F(e) := \left( \sum_{v : \mathbb{B}^{\mathrm{oh}}_{v, e} = -1} v, \quad \sum_{v : \mathbb{B}^{\mathrm{oh}}_{v, e} = +1} v \right) \in T^1(\F^{Q_0}) \times T^1(\F^{Q_0}).
\end{equation}
\end{definition}

\begin{remark}\label{rem:F-image}
By $\beta_F$, both $A_e$ and $B_e$ lie in $T^1(\F^{Q_0}) = \F^{Q_0}$. For a hyperedge $e$ with more than one source-side or target-side vertex, $A_e$ or $B_e$ is a non-trivial linear combination of vertex vectors, neither a pure tensor nor an image of $\Sym_k$. Hence $F(\Hbb_o)$ does not follow any of the standard constructions (1)--(6), but lies in the general framework of Definition \ref{def:tensor-hg}. Theorem \ref{thm:defect-vanishing}, which holds only under standard constructions, does not apply directly to $F(\Hbb_o)$, but the general theorems (Theorem \ref{thm:rank-nullity} and Theorem \ref{thm:tensor-dim}) do.
\end{remark}

\begin{theorem}\label{thm:comparison-oh}
For an oriented hypergraph $\Hbb_o = (Q_0, Q_1, \mathbb{B}^{\mathrm{oh}})$, the tensor cycle space of $F(\Hbb_o)$ coincides with the kernel of the oriented hypergraph incidence matrix:
\begin{equation}\label{eq:oh-comparison}
\Zsp(F(\Hbb_o)) = \Ker(\mathbb{B}^{\mathrm{oh}}).
\end{equation}
\end{theorem}

\begin{proof}
By the definition of $\beta_F$, for each $e \in Q_1$, we have
$$\partial_\beta(\mathbf{1}_e) = B_e - A_e = \sum_{v : \mathbb{B}^{\mathrm{oh}}_{v, e} = +1} v - \sum_{v : \mathbb{B}^{\mathrm{oh}}_{v, e} = -1} v = \sum_{v \in Q_0} \mathbb{B}^{\mathrm{oh}}_{v, e} \cdot v.$$
By the $\F$-linearity of $\partial_\beta$, for any $\xi = \sum_e a_e \mathbf{1}_e \in \F^{Q_1}$,
$$\partial_\beta(\xi) = \sum_e a_e \sum_v \mathbb{B}^{\mathrm{oh}}_{v, e} v = \sum_v \left(\sum_e \mathbb{B}^{\mathrm{oh}}_{v, e} a_e\right) v = \sum_v (\mathbb{B}^{\mathrm{oh}} \xi)_v \cdot v.$$
Since $Q_0$ is a basis of $\F^{Q_0}$, the right side vanishes if and only if $\mathbb{B}^{\mathrm{oh}} \xi = 0$. Hence $\Zsp(F(\Hbb_o)) = \Ker(\partial_\beta) = \Ker(\mathbb{B}^{\mathrm{oh}})$.
\end{proof}

\begin{corollary}\label{cor:oh-dim-formula}
The kernel dimension of the incidence matrix $\mathbb{B}^{\mathrm{oh}}$ of an oriented hypergraph $\Hbb_o$ is
$$\dim_{\F} \Ker(\mathbb{B}^{\mathrm{oh}}) = |Q_1| - |\Vmacro(F(\Hbb_o))| + \cmacro(F(\Hbb_o)) + \delta(F(\Hbb_o)).$$
\end{corollary}

\begin{proof}
By Theorem \ref{thm:comparison-oh}, $\dim_\F \Ker(\mathbb{B}^{\mathrm{oh}}) = \dim_\F \Zsp(F(\Hbb_o))$, and applying Theorem \ref{thm:tensor-dim} to the right side gives the claim.
\end{proof}

\begin{example}\label{ex:oh-defect}
Whereas $\delta(\Hbb) = 0$ for $\Hbb$ following the standard constructions (1)--(6) by Theorem \ref{thm:defect-vanishing}, examples with $\delta(F(\Hbb_o)) > 0$ arise naturally within the oriented hypergraph framework. We give one such example.

Let $Q_0 = \{a, b, c, d\}$, $Q_1 = \{e_0, e_1, e_2, e_3\}$, and define an oriented hypergraph $\Hbb_o$ as follows. Take $d$ as the unique source-side vertex of every $e_i$ ($\mathbb{B}^{\mathrm{oh}}_{d, e_i} = -1$), and target-side vertex sets
$$e_0 : \{c\}, \quad e_1 : \{a, c\}, \quad e_2 : \{b, c\}, \quad e_3 : \{a, b, c\}$$
(with $\mathbb{B}^{\mathrm{oh}}_{v, e_i} = +1$ for these vertices and $0$ otherwise). For each $e_i$, the source side $\{d\}$ and the target-side vertex set are disjoint, so $\Hbb_o$ is indeed an oriented hypergraph in the sense of Definition \ref{def:oriented-hg}. The correspondence $F$ of Definition \ref{def:natural-F} yields
\begin{align*}
& \beta_F(e_0) = (d,\ c), \quad \beta_F(e_1) = (d,\ a + c), \\ 
&\beta_F(e_2) = (d,\ b + c), \quad \beta_F(e_3) = (d,\ a + b + c).
\end{align*}

The associated macrograph $\Hmacro$ of $F(\Hbb_o)$ is a star with source tensor $d$ at the centre and four distinct target tensors $c, a + c, b + c, a + b + c$ as leaves, so $|\Vmacro| = 5$, $\cmacro = 1$, and there is no topological cycle. The tensor incidence operator gives
\begin{align*}
& \partial_\beta(\mathbf{1}_{e_0}) = c - d, \quad \partial_\beta(\mathbf{1}_{e_1}) = a + c - d, \\
& \partial_\beta(\mathbf{1}_{e_2}) = b + c - d, \quad \partial_\beta(\mathbf{1}_{e_3}) = a + b + c - d.
\end{align*}
Taking differences, $\Img(\partial_\beta) = \Span_\F\{c - d,\ a,\ b\}$, so $\rank_\F(\partial_\beta) = 3$. Applying Corollary \ref{cor:tensor-defect-rank-drop},
$$\delta(F(\Hbb_o)) = (|\Vmacro| - \cmacro) - \rank_\F(\partial_\beta) = (5 - 1) - 3 = 1.$$

We extract a generator of the cycle space explicitly. For $\xi := \mathbf{1}_{e_0} - \mathbf{1}_{e_1} - \mathbf{1}_{e_2} + \mathbf{1}_{e_3} \in \F^{Q_1}$,
$$\partial_\beta(\xi) = (c - d) - (a + c - d) - (b + c - d) + (a + b + c - d) = 0,$$
which corresponds to the case $w_0 = d$, $r = 4$, $(\alpha_i) = (1, -1, -1, 1)$ of Proposition \ref{prop:star-defect}. The tensor dimension formula \eqref{eq:tensor-dim} gives $\dim_\F \Zsp(F(\Hbb_o)) = 4 - 5 + 1 + 1 = 1$, so $\Zsp(F(\Hbb_o)) = \Span_\F\{\xi\}$, and via Theorem \ref{thm:comparison-oh} this simultaneously identifies $\Ker(\mathbb{B}^{\mathrm{oh}}) = \Span_\F\{\xi\}$. 
\end{example}

\begin{remark}\label{rem:defect-meaning}
The four target-side vertex sets $\{c\}, \{a, c\}, \{b, c\}, \{a, b, c\}$ of Example \ref{ex:oh-defect} are precisely $X \cap Y, X, Y, X \cup Y$ for $X = \{a, c\}$ and $Y = \{b, c\}$. The linear relation $\partial_\beta(\xi) = 0$ supporting the cycle $\xi$ reduces, after the source tensor $d$ cancels out, to the inclusion-exclusion identity for the indicator vectors on the target side:
$$\mathbf{1}_X + \mathbf{1}_Y = \mathbf{1}_{X \cup Y} + \mathbf{1}_{X \cap Y}.$$
Under the standard constructions (1)--(6), Lemma \ref{lem:Vmacro-nonzero-indep} ensures the $\F$-linear independence of non-zero boundary tensors, so such dependencies cannot arise and $\delta = 0$ as in Theorem \ref{thm:defect-vanishing}.

To realize $\delta(\mathcal{H}) > 0$, one needs incidence data bundling multiple vertices. Example \ref{ex:alg-cycle} achieved $\delta = 1$ by permitting signed linear combinations $a - b$ in the target tensors. Example \ref{ex:oh-defect}, by contrast, realizes $\delta > 0$ within the oriented hypergraph framework, where incidence data is restricted to vertex-set indicator vectors (entries $0$ or $1$). This shows that the defect invariant operates non-trivially even within this existing class.
\end{remark}

\begin{lemma}\label{lem:01-affine}
In a vector space over a field $\F$, an \textbf{indicator vector} is a vector each of whose entries is $0$ or $1$. Any family of three or fewer distinct indicator vectors is affinely independent over $\F$.
\end{lemma}

\begin{proof}
A family of one vector is trivially affinely independent.

For two distinct indicator vectors $x_1, x_2$, an affine dependence relation $\alpha_1 x_1 + \alpha_2 x_2 = 0$ with $\alpha_1 + \alpha_2 = 0$ gives $\alpha_1 (x_1 - x_2) = 0$. Since $x_1 \neq x_2$, $\alpha_1 = \alpha_2 = 0$ and the family is affinely independent.

Suppose three distinct indicator vectors $x_1, x_2, x_3$ are affinely dependent. We derive a contradiction. There exists a non-trivial $(\alpha_i)$ with $\sum_i \alpha_i x_i = 0$ and $\sum_i \alpha_i = 0$. If some $\alpha_i$ is zero, the relation reduces to a non-trivial affine dependence among the remaining two, contradicting the two-vector case. So we may assume $\alpha_i \neq 0$ for $i = 1, 2, 3$. From $\alpha_2 \neq 0$ and $\alpha_1 + \alpha_3 = -\alpha_2$,
$$x_2 = t\, x_1 + (1 - t)\, x_3, \quad t := -\alpha_1/\alpha_2.$$
Since $x_1 \neq x_3$, there is some coordinate $v$ at which $(x_1)_v$ and $(x_3)_v$ differ, one being $0$ and the other $1$. The value $(x_2)_v$ at this coordinate is $t$ or $1 - t$, but the indicator property forces $(x_2)_v \in \{0, 1\}$, so $t \in \{0, 1\}$. However, $t = 0$ gives $x_2 = x_3$ and $t = 1$ gives $x_2 = x_1$, both contradicting the distinctness of $x_1, x_2, x_3$.
\end{proof}

\begin{proposition}\label{prop:oh-star-minimality}
Let $\Hbb_o$ be an oriented hypergraph such that, in $F(\Hbb_o)$, all hyperedges share a common non-empty source-side vertex set $S$ and the target-side vertex sets $H_1, \ldots, H_r$ ($r = |Q_1|$) are distinct. Then the associated macrograph is a star, and $\delta(F(\Hbb_o)) > 0$ if and only if the indicator vectors $\mathbf{1}_{H_1}, \ldots, \mathbf{1}_{H_r}$ are affinely dependent over $\F$. In particular, $\delta(F(\Hbb_o)) > 0$ implies $r \geq 4$.
\end{proposition}

\begin{proof}
The source tensor of each hyperedge is $\mathbf{1}_S$, and the target tensor is $\mathbf{1}_{H_i}$. From $S \neq \emptyset$ and $S \cap H_i = \emptyset$, we have $\mathbf{1}_S \neq \mathbf{1}_{H_i}$. Since the $H_i$ are distinct, the associated macrograph is a star with centre $\mathbf{1}_S$ and $r$ leaves $\mathbf{1}_{H_i}$. Hence $|\Vmacro| = r + 1$, $\cmacro = 1$, and Corollary \ref{cor:tensor-defect-rank-drop} gives
\begin{align*}
\delta(F(\Hbb_o)) & = (|\Vmacro| - \cmacro) - \rank_{\F}(\partial_{\beta}) \\
& = r - \dim_{\F} \Span_\F\{\mathbf{1}_{H_i} - \mathbf{1}_S\mid i = 1, \ldots, r\}.
\end{align*}
Therefore $\delta(F(\Hbb_o)) > 0$ if and only if the $r$ vectors $\mathbf{1}_{H_i} - \mathbf{1}_S$ are linearly dependent over $\F$.

Consider a linear dependence relation $\sum_i \alpha_i (\mathbf{1}_{H_i} - \mathbf{1}_S) = 0$. Since $H_i \subseteq Q_0 \setminus S$, the support of $\sum_i \alpha_i \mathbf{1}_{H_i}$ lies in $Q_0 \setminus S$ and the support of $(\sum_i \alpha_i) \mathbf{1}_S$ lies in $S$. These are disjoint. Hence the relation is equivalent to
$$\sum_i \alpha_i \mathbf{1}_{H_i} = 0 \quad \text{and} \quad \sum_i \alpha_i = 0$$
(using $S \neq \emptyset$, so $\mathbf{1}_S \neq 0$), which is precisely an affine dependence relation among the indicator vectors $\mathbf{1}_{H_1}, \ldots, \mathbf{1}_{H_r}$. This establishes the equivalence of $\delta(F(\Hbb_o)) > 0$ and the affine dependence of $\{\mathbf{1}_{H_i}\}$.

By Lemma \ref{lem:01-affine}, any three or fewer distinct indicator vectors are affinely independent, so $r \geq 4$ is neccesary for affine dependence.
\end{proof}

The correspondence $F$ embeds oriented hypergraphs faithfully into Definition \ref{def:tensor-hg} (Remark \ref{rem:F-image}), but $F(\Hbb_o)$ generally does not follow a single standard construction and belongs to an extended class allowing linear combinations of vertex vectors. Conversely, our standard constructions employ higher-degree symmetric tensors $\Sym_k$ as boundaries and do not fit into the signed-scalar framework $\{-1, 0, +1\}$. The two theories differ in three respects: (i) the incidence data is restricted to signed scalars in oriented hypergraphs, while ours allows source and target tensors of arbitrary degree; (ii) this difference enables the direct treatment of multiset hyperedges via $\Sym_k$ in our framework; (iii) the oriented hypergraph cycle space is defined as the kernel of an incidence matrix, while $\Zsp(\Hbb)$ is the kernel of a difference map within the tensor algebra, with the two connected through observation maps (Section \ref{sec:observation}). Theorem \ref{thm:comparison-oh} and Example \ref{ex:oh-defect} concretely bridge this gap. A recent extension of oriented hypergraph theory in a different direction is given in \cite{Rusnak2025}, where the Harary--Sachs theorem is generalized to integer matrices via incidence-based cycle covers; this remains within the signed-scalar incidence regime and is complementary to the tensor-valued framework developed here.


\end{document}